\documentclass{amsart}

\usepackage[mathscr]{eucal}
\usepackage{amsfonts}
\usepackage{amsmath}
\usepackage{amsthm}
\usepackage{amssymb}
\usepackage{latexsym}

\usepackage{tikz-cd}

\usepackage{graphicx}
\usepackage{color}
\usepackage[noadjust]{cite}
\usepackage{epsfig}

\allowdisplaybreaks 

\usepackage[final]{showkeys}

\newfont{\msam}{msam10}
\newtheorem{theoremintro}{Theorem}
\newtheorem{theorem}{Theorem}[section]
\newtheorem{proposition}[theorem]{Proposition}
\newtheorem{corollary}[theorem]{Corollary}
\newtheorem{lemma}[theorem]{Lemma}
\theoremstyle{definition}
\newtheorem{definition}[theorem]{Definition}
\newtheorem{remark}[theorem]{Remark}

\newtheorem{question}[theorem]{Question} 

\let\nc\newcommand

\nc{\la}{\label}

\def\bthm{\begin{theorem}}
\def\ethm{\end{theorem}}
\def\blemma{\begin{lemma}}
\def\elemma{\end{lemma}}
\def\bproof{\begin{proof}}
\def\eproof{\end{proof}}
\def\bprop{\begin{proposition}}
\def\eprop{\end{proposition}}

\def\C{\mathbb{C}}
\nc{\End}{{\rm{End}}}

\numberwithin{equation}{section}

\newcommand{\bb}{{\mathbf{b}}}

\newcommand{{\cc}}{\mathcal{C} }
\newcommand{{\sa}}{\mathcal{A}}
\newcommand{{\sd}} {\mathcal{D}}
\newcommand{\si}{\mathcal{I}}
\newcommand{{\Z}}{\mathbb{Z}}

\newcommand{\xx}{{\mathbf{x}}}
\newcommand{\yy}{{\mathbf{y}}}
\newcommand{\aab}{{\mathbf{a}}}

\newcommand{\cl}{\mathrm{cl}}

\newcommand{\GL}{\mathrm{GL}}

\newcommand{\N}{\mathbb{N}}
\newcommand{\Q}{\mathbb{Q}} 

\newcommand{\PS}[1]{}
\newcommand{\AP}[1]{}
\newcommand{\HM}[1]{}

\newcommand{\sk}{\mathrm{Sk}}

\newcommand{\fg}{\mathfrak{g}}

\begin{document}
\title[]{The Kauffman skein algebra of the torus}
\date{\today}

\author{Hugh Morton}
\address{Department of Mathematical Sciences, University of Liverpool, Peach Street, Liverpool L69 7ZL, UK}
\email{morton@liv.ac.uk}

\author{Alex Pokorny}
\address{Department of Mathematics, University of California, Riverside}
\email{apoko001@ucr.edu}

\author{Peter Samuelson}
\address{Department of Mathematics, University of California, Riverside}
\email{psamuels@ucr.edu}

\begin{abstract}
We give a presentation of the Kauffman (BMW) skein algebra of the torus. This algebra is the ``type $BCD$'' analogue of the Homflypt skein algebra of torus which was computed in earlier work of the first and third authors. In the appendix we show this presentation is compatible with the Frohman-Gelca description of the Kauffman bracket (Temperley-Lieb) skein algebra of the torus \cite{FG00}.
\end{abstract}
\maketitle

\tableofcontents

%
%


\section{Introduction}\label{sec:intro}
The \emph{skein} $\sk(M)$ of the 3-manifold $M$ is a vector space formally spanned by embedded links in $M$, modulo certain local relations which hold in embedded balls in $M$. These relations are referred to as \emph{skein relations}, and there are several versions, each of which encodes a linear relation between $R$-matrices for some quantum group. \emph{A priori}, $\sk(M)$ is just a vector space, but if $M = F \times [0,1]$, then $\sk(M)$ is an algebra, where the product is given by stacking in the $[0,1]$ direction.

In recent years there have been several results relating skein algebras of surfaces to algebras constructed using representation theoretic or geometric methods. Let us somewhat imprecisely state two of these results which are most relevant to the present paper.

\begin{theoremintro}\label{thm:past} We have the following isomorphisms.
\begin{enumerate}
	\item The Kauffman bracket skein algebra of the torus is isomorphic to the $t=q$ specialization of the $\mathfrak{sl}_2$ spherical double affine Hecke algebra. \cite{FG00}
	\item The Homflypt skein algebra of the torus is isomorphic to the $t=q$ specialization of the elliptic Hall algebra of Burban and Schiffmann. \cite{MS17}
\end{enumerate}
\end{theoremintro}

The double affine Hecke algebra was defined by Cherednik (for general $\fg$, see, e.g. \cite{Che05}) as part of his proof of the Macdonald conjectures. The elliptic Hall algebra was defined by Burban and Schiffmann, who showed it is the ``generic'' Hall algebra of coherent sheaves over an elliptic curve over a finite field. This algebra can also be viewed as a DAHA (the spherical $\mathfrak{gl}_\infty$ DAHA) by work of  Schiffmann and Vasserot in \cite{SV11}. These algebras have appeared in many guises and have found many applications over the years, a small sample of which includes the following:
\begin{itemize} 
	\item Cherednik's proof of the Macdonald conjectures \cite{Che95},  
	\item algebraic combinatorics (proofs and generalizations of the shuffle conjecture \cite{CM18, BGLX16}), 
	\item algebraic geometry (convolution operators in equivariant $K$-theory of $\mathrm{Hilb}_n$ \cite{SV13Hilb}, \cite{Neg15}),
	\item knot theory (computations of Homflypt homology of positive torus knots \cite{EH19}, \cite{Mel17}), 
	\item mathematical physics (a mathematical version of the AGT conjectures \cite{SV13}), 
	\item enumerative geometry (counting stable Higgs bundles \cite{Sch16}),
	\item categorification (the Hochschild homology of the quantum Heisenberg category \cite{CLLSS18}).
\end{itemize}

These results suggest that it is worthwhile to understand skein algebras of surfaces, and in particular of the torus, for various versions of the skein relations. 
In this paper we give a presentation of the algebra $\sd$ associated to the torus by the Kauffman\footnote{To clarify a possible point of confusion, the Kauffman \emph{bracket} skein relations are defined in the Appendix. Computations with the Kauffman bracket skein relations are much simpler because ``all crossings can be removed.'' See also Remark \ref{rmk:sobig}.} skein relations in equations \eqref{eq:sk1} and \eqref{eq:sk2}; these are the relations used to define the BMW algebra \cite{BW89, Mur87}, which is the centralizer algebra of tensor powers of the defining representation of quantum groups in type $BCD$.

\begin{theoremintro}[see Theorem \ref{thm:presentation}] \label{thm:main}
	The Kauffman skein algebra $\sd$ of the torus has a presentation with generators $D_{\xx}$ for $\xx \in \Z^2$, with relations
	\begin{align}
	[D_\xx, D_\yy] &= (s^d - s^{-d})\left( D_{\xx+\yy} - D_{\xx-\yy}\right),\\
	D_\xx&= D_{-\xx}.
	\end{align}
\end{theoremintro}

This theorem implies that the subspace $\mathcal L := \mathrm{span} \{D_\xx\}$ is a Lie algebra, since it is preserved by the Lie bracket $[a,b] := ab - ba$. The following corollary is also true in the Homflypt case; however, the proofs in both cases are purely computational.
\begin{corollary}
	The skein algebra $\sd$ is the enveloping algebra of a Lie algebra (the subspace spanned by the $D_\xx$).
\end{corollary}

If $\xx \in \Z^2$ is primitive (i.e. its entries are relatively prime), then $D_\xx$ is the (unoriented) simple closed curve on the torus of slope $\xx$. However, the definition for nonprimitive $\xx$ is more complicated, and in some sense is the key difficulty in the paper. The strategy for the proof of Theorem \ref{thm:main} follows the ideas of \cite{MS17}, which in turn was motivated by some proofs in \cite{BS12}. Using skein theoretic calculations and some identities involving idempotents in the BMW algebra, we prove two special cases of these relations (see Proposition \ref{prop:specialcases}). We then use the natural $SL_2(\Z)$ action and an induction argument to show that all the claimed relations follow from these special ones. 

We also note that the $D_\xx$ come from elements $B_k \in BMW_k$ 
that are analogues of the power sum elements $P_k \in H_k$ in the Hecke algebra that were found and studied by Morton and coauthors. In Theorem \ref{thm:aid}, which may be of independent interest, we show that the $B_k$ satisfy a ``fundamental identity'' which is an analogue of an identity for the $P_k$ found by Morton and coauthors. We aren't aware of a previous appearance of this identity in the BMW case.

In both of the statements in Theorem \ref{thm:past}, the skein algebra was computed explicitly and then shown to be isomorphic to a DAHA-type algebra that had already appeared in the literature.  However, at present we are not aware of a previous appearance of the algebra described in Theorem \ref{thm:main}. This motivates a number of questions, which we state somewhat vaguely:
\begin{question}
	Is there a $t$-deformation of the algebra $\sd$? Does it have an interpretation in terms of Hall algebras, geometry of Hilbert schemes, double affine Hecke algebras, or categorification?
\end{question}

We do note\footnote{We thank Negut for pointing this out.}  that this algebra is unlikely to literally be the Hall algebra of a category because it does not have a $\Z^2$ grading\footnote{The Hall algebra is graded by $K$-theory, and in many cases of interest this is a free abelian group of rank 2 or greater.}. In general, the Kauffman skein of a 3-manifold $M$ is graded by $H_1(M, \Z/2)$, so the skein algebra $\sd$ of the torus has a $\Z/2 \oplus \Z/2$ grading.


A summary of the contents of the paper is as follows. In Section 2, we recall some background and prove some identities in the $BMW$ algebra. In Section 3, we state the main theorem and reduce its proof to 2 families of identities, which are then proved in Sections 4 and 5. The Appendix regards a classical theorem of Przytycki and also relates Theorem \ref{thm:main} to the work of Frohman and Gelca in \cite{FG00} describing the Kauffman \emph{bracket} skein algebra of the torus.

\noindent \textbf{Acknowledgements:} This paper was initiated during a Research in Pairs stay by the first and third authors at Oberwolfach in 2015, and we appreciate their support and excellent working conditions. Further work was done at conferences at the Banff International Research Station and the Isaac Newton Institute, and the third author's travel was partially supported by a Simons travel grant. The authors would like to thank A. Beliakova, C. Blanchet, C. Frohman, F. Goodman,  A.  Negut, A. Oblomkov, and A. Ram for helpful discussions.

\section{The BMW algebra and symmetrizers}

In this section we define the Kauffman skein algebra, the BMW algebra $BMW_k$, and elements $B_k \in BMW_k$ which are fundamental in the paper. We relate the closure of $B_k$ in the annulus to elements $D_k$ defined using power series, and we prove some identities involving the $D_k$ that will be needed in later sections.

\subsection{Notation and definitions}
Here we will record some notation and definitions. Everywhere our base ring will be
\[R := \Q (s,v).\]
We will use standard notation for quantum integers: 
	\[ \{n\} := s^n - s^{-n},\quad \quad \{n\}^+ := s^n + s^{-n},\quad \quad 
	[n] := \{ n \} / \{ 1 \}.
	\]
	

\begin{definition}
	The Kauffman skein $\sd(M)$ of a 3-manifold $M$ is the vector space spanned by framed unoriented links in $M$, modulo the Kauffman skein relations:
	\begin{align}\label{eq:sk1}
	\vcenter{\hbox{\includegraphics[height=2cm]{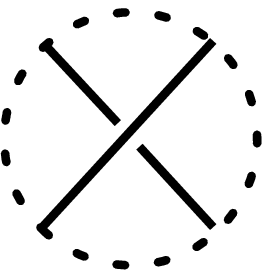}}} \,\, - \,\, \vcenter{\hbox{\includegraphics[height=2cm]{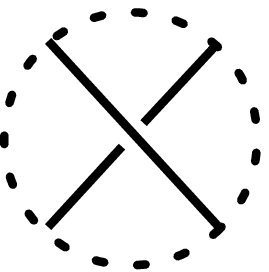}}} \quad = \quad (s-s^{-1}) \left( \vcenter{\hbox{\includegraphics[height=2cm]{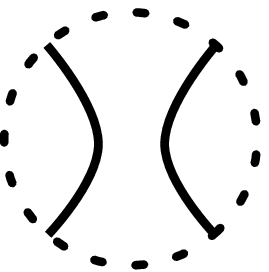}}} \right. \,\, - \,\, \left. \vcenter{\hbox{\includegraphics[height=2cm]{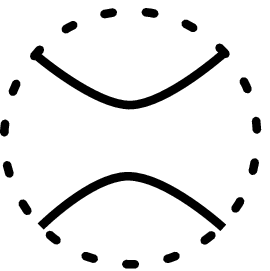}}} \right),
	\end{align}
	\begin{align}\label{eq:sk2}
 	\vcenter{\hbox{\includegraphics[height=2cm, keepaspectratio]{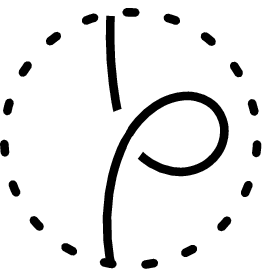}}} \quad = \quad v^{-1} \,\, \vcenter{\hbox{\includegraphics[height=2cm, keepaspectratio]{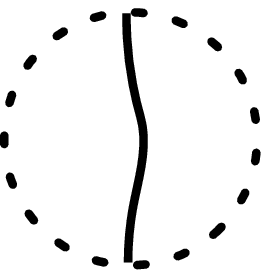}}} \,\, , \qquad \qquad \vcenter{\hbox{\includegraphics[height=2cm, keepaspectratio]{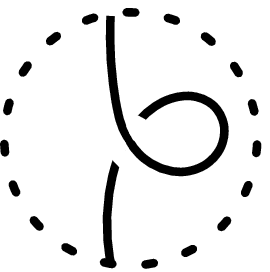}}} \quad = \quad v \,\, \vcenter{\hbox{\includegraphics[height=2cm, keepaspectratio]{frameresolution.eps}}}.
	\end{align}

\end{definition}

In these relations we have used the ``blackboard framing'' convention, so that the framing of each arc is perpendicular to the paper. A short computation involving the skein relations leads to the following value for the unknot:
\[
\vcenter{\hbox{\includegraphics[height=2cm, keepaspectratio]{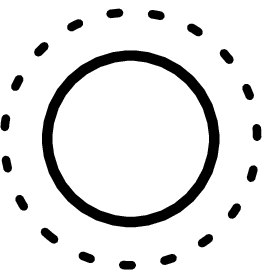}}} = 1- \frac{v - v^{-1}}{s - s^{-1}} =: \delta.
\]
We will abbreviate the notation for the thickened torus and solid torus as follows:
\begin{itemize}
	\item $\sd := \sd(T^2\times [0,1])$ is the skein algebra of the thickened torus,
	\item $\cc := \sd(S^1 \times [0,1]^2)$ is the skein of the solid torus.
\end{itemize}
\begin{remark}\label{rem:t2action}
	Since the boundary of the solid torus is the torus, there is an action of the algebra $\sd$ on the vector space $\cc$. This action depends on the identification of $S^1 \times S^1$ with the boundary of the solid torus, and the convention we use is described in the second paragraph of Section \ref{sec:angled}. We will refer to this action in Sections \ref{sec:lifting} and \ref{sec:angled}, where we give a partial description of this action.

\end{remark} 
	 

\subsection{The Hecke and BMW Algebras}\label{sec:bmw}

The Birman-Murakami-Wenzl algebra $BMW_n$ is the $R$-algebra 
generated by elements $\sigma_1, \cdots, \sigma_{n-1}, h_1, \cdots, h_{n-1}$ subject to certain relations, which can be found in \cite{BW89, Mur87 }. They showed that the BMW algebra is isomorphic to the Kauffman algebra (see also \cite{Mor10, GH06}), which consists of tangles in the square with $n$ points on both the top and bottom edges, with orientation removed, modulo the Kauffman skein relations, where the multiplication is stacking diagrams vertically. By convention, the diagram of the product $XY$ is obtained by joining the top edge of $X$ with the bottom edge of $Y$. This is called the \textit{relative skein module} of the square with $2n$ boundary points. Under this diagrammatic framework, the generators are represented by:
\[
\sigma_i = \vcenter{\hbox{\includegraphics[height=3cm, keepaspectratio]{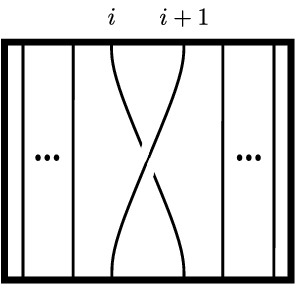}}} \qquad \qquad h_i = \vcenter{\hbox{\includegraphics[height=3cm, keepaspectratio]{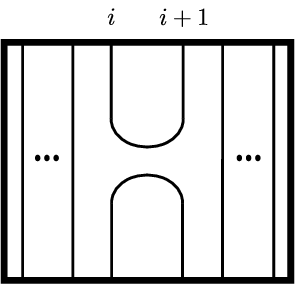}}}.
\]
A general theory of relative skein modules exists, and we will discuss another example of such a skein in Section \ref{sec:perpendicular}. 

The Hecke algebra $H_n$ is similarly defined as the relative Homflypt skein module of the square with $2n$ boundary points. As an algebra, it is generated by $\sigma_1,\cdots,\sigma_{n-1}$, modulo the quadratic relations $\sigma_i - \sigma^{-1}_i = s - s^{-1}$. We obtain the following algebra projection by setting $h_i = 0$:
\[
\pi_n: BMW_n \to H_n.
\] 

In this subsection we first recall a very useful result of Beliakova and Blanchet in \cite{BB01} relating the Hecke algebra $H_n$ and the BMW algebra, which we use to prove an identity involving the $D_k$. Let  $I_n$ be the kernel of $\pi_n$.

\begin{theorem}[{\cite[Thm. 3.1]{BB01}}]\label{thm:bb01}
	There exists a unique (non-unital\footnote{To be clear, when we say ``non-unital algebra map,'' we mean $s_n(x+y)=s_n(x) + s_n(y)$ and $s_n(xy) = s_n(x)s_n(y)$, but that $s_n(1)$ is not the identity in $BMW_n$, but instead is another idempotent.}) algebra map 
	\[s_n: H_n \to BMW_n
	\] 
	which is a section of $\pi_n$ such that 
	\begin{equation*}
	s_n(x)y = 0 = y s_n(x)\quad \mathrm{for} \,\, x \in H_n,\,\, y \in I_n.
	\end{equation*}
\end{theorem}

The Hecke algebra has minimal central idempotents $z_\lambda \in H_n$ indexed by partitions of $n$:
\[
\{\lambda = (\lambda_1,\cdots,\lambda_k)\mid \lambda_j \in \mathbb{N}_{> 0},\, \lambda_i \geq \lambda_{i+1},\,\, \sum \lambda_i = n\}. 
\]
\\
The section $s_n$ transfers these to minimal central idempotents in the BMW algebra. In particular, if $z_{(n)} \in H_n$ is the symmetrizer, then $s_n(z_{(n)}) = f_n$ is the symmetrizer in the BMW algebra. We use these symmetrizers to make the key definition of the paper. 

Given an element $x \in BMW_n$, let $\hat x \in \cc$ be the annular closure of a diagram representing $x$:
\[
\hat{x} = \vcenter{\hbox{\includegraphics[height=4cm, keepaspectratio]{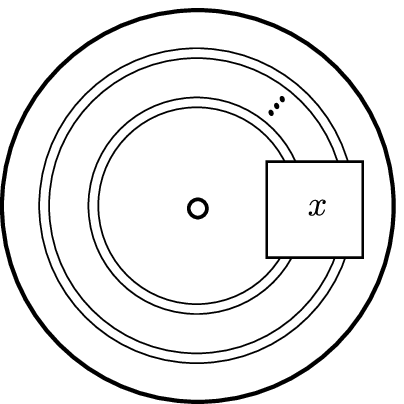}}}.
\]
\begin{definition}\label{def:dk}
Define elements $D_k \in \cc$ using the following equality of power series:
\begin{equation}\label{eq:dk}
\sum_{k=1}^\infty \frac{D_k}{k} t^k:=\ln\left(1 + \sum_{i \geq 1} \hat f_i t^i\right).
\end{equation}	
Note the right hand side makes sense because $\cc$ is a commutative algebra.
\end{definition}

It will also be useful for us to relate the Beliakova-Blanchet section $s_n$ to the inclusions between Hecke and $BMW$ algebras of different ranks. We write the standard inclusions between these algebras as follows:
 \begin{align*}
	\iota^A_{m,n}&: H_m \otimes H_n \to H_{m+n},\\
	\iota^B_{m,n}&: BMW_m \otimes BMW_n \to BMW_{m+n}.
 \end{align*}
(These inclusions are induced from similar inclusions between braid groups.) The following lemma could be made more precise, but we won't need to do so for our purposes.
\begin{lemma}\label{lemma:inc}
	For $x \in H_m$ and $y \in H_n$ we have the following equality for some $a \in \ker(\pi_{m+n})$:
	\[
	\iota^B_{m,n}(s_m(x)\otimes s_n(y)) = s_{m+n}(\iota^A_{m,n}(x \otimes y)) + a.
	\]
\end{lemma}
\begin{proof}
Consider the following diagram:
\begin{center}
\begin{tikzcd}
	H_m \otimes H_n \arrow[d,"s_m\otimes s_n"] \arrow[r, "\iota_{m,n}^A"] 
	& H_{m+n} \arrow[d, "s_{m+n}"]\\
	BMW_m \otimes BMW_n \arrow[r, "\iota^B_{m,n}"] 	\arrow[d,"\pi_m\otimes \pi_n"]  & BMW_{m+n}\arrow[d, "\pi_{m+n}"]\\
	H_m \otimes H_n \arrow[r, "\iota_{m,n}^A"] & H_{m+n}
\end{tikzcd}.
\end{center}
The bottom square of this diagram is commutative since both horizontal arrows are induced from the same map on braid groups. Also, the big rectangle is commutative, since $s_k$ is a section of $\pi_k$. This implies the top square is commutative up to an element in the kernel of $\pi_{m+n}$, which is equivalent to the claimed statement.
\end{proof}

\subsection{Relating $D_k$ to hook idempotents}


Write $\widehat{f}_n \in \cc$ for the annular closure of the symmetriser $f_n \in BMW_n$ and similarly $\widehat{e}_n \in \cc$ for that of the antisymmetriser. We will also use the notation $\widehat{Q}_{(i|j)} \in \cc$ for the image in $\cc$ of the class $[s_n(z_{(i|j)})]$ corresponding to the hook $\lambda =(i|j)$, which has an arm of length $i$ and a leg of length $j$, and whose Young diagram contains $i+j+1$ cells. An example of a hook partition is
\[
\lambda = \vcenter{\hbox{\includegraphics[height=3cm, keepaspectratio]{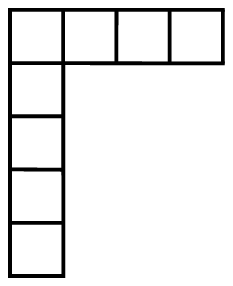}}}
\]
\\
where $\lambda = (3|4)$. Then we have as a special case $\widehat{f}_n=\widehat{Q}_{(n-1|0)}$ and $\widehat{e}_n=\widehat{Q}_{(0|n-1)}$.

We can use the decomposition of the product of $\widehat{f}_i$ and $\widehat{e}_j$ as a sum of hooks to analyse the element $D_k$ in terms of hooks.
Write $F(t), E(t)$ for the formal generating functions 
\[
F(t):=1+\sum_{i=1}^\infty \widehat{f}_i t^i, \ E(t)=1+\sum_{i=1}^\infty \widehat{e}_if t^i.
\] 
Recall that we have defined $D_k$ by 
\[
\sum_{k=1}^\infty \frac{D_k}{k} t^k:=\ln(F(t)).
\] 
Taking the formal derivative of these series gives the alternative description 
\[
\sum_{k=1}^\infty D_k t^{k-1}=\frac{F'(t)}{F(t)}.
\]

Let us also introduce the formal 2-variable hook series $S(t,u)$ by setting
\[
S(t,u):=\sum_{k,l\ge 0} \widehat{Q}_{(k|l)} t^k u^l.
\]

\begin{proposition}\label{hooks}
	The product of the two series $F(t)$ and $E(u)$ satisfies 
	\begin{equation}
	F(t) E(u) = (1+tu)\left(1+(t+u)S(t,u)\right).\label{hookeqn}
	\end{equation}
\end{proposition}

\begin{proof}
	
By Lemma \ref{lemma:ktheory} below,
we have an expansion for the product $\widehat{f}_i \widehat{e}_j$ in terms of hooks:
	\[
	\widehat{f}_i \widehat{e}_j = \widehat{Q}_{(i-1|j)} + \widehat{Q}_{(i|j-1)} + \widehat{Q}_{(i-2|j-1)} + \widehat{Q}_{(i-1|j-2)}.
	\]
	This relation is true for $i,j>1$, and also holds when one of $i$ and $j$ is equal to $1$, with the convention that $\widehat{Q}_{(k|l)}=0$ when $k<0$ or $l<0$.
	The one exception is that when $i=j=1$ there is an extra term  
	\[
	\widehat{f}_1 \widehat{e}_1 = \widehat{Q}_{(0|1)}+ \widehat{Q}_{(1|0)} +1.
	\]
	
	This accounts for the term $tu$ on the right hand side of equation \eqref{hookeqn}.  
	Otherwise the coefficient of $t^i u^j$ in $(t+u+t^2u+tu^2)S(t,u)$ is  
	\[
	\widehat{Q}_{(i-1|j)} + \widehat{Q}_{(i|j-1)} + \widehat{Q}_{(i-2|j-1)} + \widehat{Q}_{(i-1|j-2)}
	\]
	with the same convention, and so agrees with the coefficient of $t^i u^j$ in $F(t)E(u)$.
\end{proof}

\begin{lemma}\label{lemma:ktheory}
We have the following identity:
	\[
\widehat{f}_i \widehat{e}_j = \widehat{Q}_{(i-1|j)} + \widehat{Q}_{(i|j-1)} + \widehat{Q}_{(i-2|j-1)} + \widehat{Q}_{(i-1|j-2)}.
\]
\end{lemma}
\begin{proof}
There is a Chern character map $K_0(C) \to HH_0(C)$, where $C$ is the BMW category, and an algebra map $HH_0(C) \to \cc$. The identity in question is in the image of the Chern character map: the LHS is (the image of) an object in $C$ and the RHS represents its decomposition into a sum of simple objects. This implies the coefficients must be natural numbers, and in particular are independent of $q$. We may therefore compute them for any value where the idempotents are defined, i.e. at $q=1$. This was done in a paper of Koike and Terada \cite{KT87}.
\end{proof} 

In  Proposition \ref{hooks} we can take $u=-t$ to get the relation $F(t)E(-t)=1-t^2$.

\begin{theorem}\label{thm:hooksum}
	We have 
	\[
	D_k=\sum_{i+j+1=k} (-1)^j \widehat{Q}_{(i|j)} + c_k,
	\] 
	where 
	\begin{equation}\label{eq:ck}
	c_k := 
	\begin{cases}
	0 & k \textrm{ odd} \\
	-1 & k \textrm{ even}.
	\end{cases}
	\end{equation}
\end{theorem}
\begin{proof}
	Differentiate equation \eqref{hookeqn} with respect to $t$. Then
	\[
	F'(t)E(u)=u\left(1+(t+u)S(t,u)\right) +(1+tu)\left( S(t,u)+(t+u)\frac{\partial S}{\partial t}\right).
	\]
	
	Divide this by $F(t)E(u)$ and set $u=-t$.
	Then 
	\[
	\frac{F'(t)}{F(t)}=\frac{-t +(1-t^2)S(t,-t)}{1-t^2} = S(t,-t)-\frac{t}{1-t^2}.
	\]
	Now $D_k$ is the coefficient of $t^{k-1}$ in this series, which is as claimed, with $c_k=-1$ when $k$ is even, and $c_k=0$ when $k$ is odd.
\end{proof}

\subsection{Lifting identities from the Hecke algebra}\label{sec:lifting}
We now recall elements in the Hecke algebra which were studied by Morton and coauthors, and were useful in \cite{MS17}. We will show that the section of \cite{BB01} maps these elements $B_k \in BMW_k$ whose closures are (essentially) the elements $D_k$. We use this to ``lift'' some identites involving these elements from the Hecke algebra to the BMW algebra.

\begin{definition}\label{def:pk}
	Let $P_n \in H_n$ be given by the formula
	\[
	P_n = \frac{s-s^{-1}}{s^n-s^{-n}} \sum_{i=0}^{n-1} \sigma_1 \cdots \sigma_i \sigma_{i+1}^{-1}\cdots \sigma_{n-1}^{-1}
	\]
	where, by convention, the first term is $\sigma_1^{-1}\cdots \sigma_{n-1}^{-1}$ and the last term is $\sigma_1 \cdots \sigma_{n-1}$.
\end{definition}

\HM{Do we get a nice expression for $B_k$ or $D_k$ that is a good extension of this?}
\PS{That's a good question, and the answer isn't so clear. The Beliakova-Blanchet section $s_n: H_n \to BMW_n$ from the Hecke algebra to the BMW algebra is defined (at least on braid generators) as $s_n(x) = e x e$, for some idempotent $e$. However, the definition of this idempotent seems somewhat complicated in their paper, and it's defined inductively (although I haven't read it carefully). $B_n$ is define as $B_n= s_n(P_n) = e P_n e$, but I don't know a nicer expression for it} 

Recall that the Hochschild homology $HH_0(A)$ of an algebra $A$ is the vector space defined by 
\[
HH_0(A) := \frac{A} {\mathrm{span}_{R}\,  \{ab - ba \mid a,b \in A\} }.
\]
This is useful for us because of the following lemma.
\begin{lemma}\label{lemma:close}
	There is an injective linear map $cl: HH_0(H_n) \to Homflypt(S^1\times D^2)$ from the Hochschild homology of the Hecke algebra into the Homflypt skein algebra of the annulus. There is a (noninjective) linear map $cl: HH_0(BMW_n) \to BMW(S^1\times D^2)$ from the Hochschild homology of the BMW algebra to the Kauffman skein algebra of the annulus.
\end{lemma}
\begin{proof}
	The existence of both linear maps follows from the fact that the Hecke algebra (BMW algebra) is isomorphic to the algebra of tangles in the disk modulo the Homflypt (Kauffman) skein relations. The injectivity in the Homflypt case is well known (see, e.g. \cite[Thm.\ 3.25]{CLLSS18} for a stronger statement). 
\end{proof}
\begin{remark}
	In the BMW case this closure map does not induce an injection from $HH_0(BMW_n)$. For example, the class of the ``cup-cap'' element $h_1 \in BMW_2$ is not a multiple of the class of the identity since $BMW_2$ is commutative, but $1-\delta^{-1}h_1$  is sent to zero in the skein algebra, where
\[
\delta = 1- \frac{v - v^{-1}}{s - s^{-1}}
\]
is the value of the unknot.
\end{remark}

We write $[a] \in HH_0(A)$ for the class of an element $a \in A$. 
If we expand the class $[P_k]$ in terms of the classes of the central idempotents $z_\lambda$, we obtain the following identity involving the so-called ``hook'' idempotents:
\begin{equation}\label{eq:pkidem}
[P_k]=\sum_{i+j+1=k} (-1)^j [z_{(i|j)}].
\end{equation}
(This follows e.g. from \cite[Eq. (4.2)]{MS17} and the injectivity statement in Lemma \ref{lemma:close}.)
This equation explains the notation for $P_k$, since this identity is satisfied by the power sums in the ring of symmetric functions. 

\begin{definition}\label{def:bk}
The image in the BMW algebra of the power sum element defined above is denoted
\begin{equation}\label{eq:bk}
B_k := s_k(P_k) \in BMW_k.
\end{equation}
\end{definition}


To separate notation between the Hecke and BMW algebras, we write $Q_\lambda := s_n(z_\lambda)$ for the image in the BMW algebra of idempotents in the Hecke algebra. The following identity will be used to relate the closure of $B_k$ to the element $D_k$ defined using the power series \eqref{eq:dk}.
\begin{corollary}\label{cor:skpkhook}
	The $B_k$ satisfy the following identities: 
	\begin{align}
	[B_k]&=\sum_{i+j+1=k} (-1)^j [Q_{(i|j)}] \label{eq:hooksum},\\
	D_k &= \hat{B}_k + c_k \label{eq:dvsb} 
	\end{align}
	where $c_k$ is defined as in equation \eqref{eq:ck}.
\end{corollary}
\begin{proof}
	The first identity follows immediately from Theorem \ref{thm:bb01} and from equation \eqref{eq:pkidem}, and the second follows from the first and Theorem \ref{thm:hooksum}.
\end{proof}

We will also need a more general identity which will be used to describe ``part'' of the action of $\sd$ on $\cc$  described Remark \ref{rem:t2action}. In the Hecke/Homflypt case this was described completely in \cite{MS17}, and here we slightly rephrase some of these results in terms of Hochschild homology. 

Let $k \in \N_{\geq 1}$ and $n\in  \N_{\geq 2}$, and let $m \in \Z$ be relatively prime to $n$. Here we write an element $T_{km,kn} \in H_{kn}$ in the Hecke algebra whose closure is  the image of the $(km,kn)$ torus link in the annulus.
\begin{align*}
\alpha_{i,j} &:= \sigma_{i}\sigma_{i+1}\cdots \sigma_{j-1}\\
T_{km,kn} &:= \left( \alpha_{k,n}\alpha_{k-1,n-1}\cdots \alpha_{1,n-k}\right)^m
\end{align*}
Let $H_k \hookrightarrow H_{k+1}$ be the standard inclusion of Hecke algebras. We abuse notation and write $P_k \in H_n$ for the image of $P_k$ under compositions of these inclusions. We next define elements in the Hecke algebra whose closures are positive torus links colored by $P_k$ (the factor of $v$ is to compensate for framing).
\begin{equation}
\tilde T_{km,kn} := v^{-km} P_k T_{km,kn}
\end{equation}
The following identity is a special case of \cite[Thm. 4.6, (4.3)]{MS17} when the partitions $\lambda$ and $\mu$ are  empty.
\begin{lemma}
	For $n \geq 1$, in $HH_0(H_{kn})$, we have the following identity:
	\[
	[\tilde T_{km,kn}] = \sum_{a+b+1=kn} (-1)^b v^{-km} s^{km(a-b)} [z_{(a\mid b)}].
	\]
\end{lemma}
\AP{Is this identity meaningful for $n=0$?} \PS{No, added a comment}
\begin{proof}
	Since the closure map $HH_0(H_n) \to Homflypt(S^1\times D^2)$ is injective, we can transfer the skein-theoretic identity \cite[Thm. 4.6, (4.3)]{MS17} into the claimed identity in $HH_0(H_n)$.
\end{proof}

We then have the following corollary which describes part of the action of the BMW skein algebra of the torus on the BMW skein module of the annulus.\\
 \PS{Check carefully -- are the BB sections compatible with inclusions?}
\begin{corollary}\label{cor:projection}
	In $HH_0(BMW_{kn})$, we have the identity
	\[
	[s_{kn} (\tilde T_{km,kn})] = \sum_{a+b+1=kn} (-1)^b v^{-km} s^{km(a-b)}  [Q_{(a|b)}].
	\]
\end{corollary}

\begin{remark}
	Arun Ram pointed out that some of the identities involving the $B_k$ and $D_k$ are reminiscent of identities in \cite{DRV13, DRV14}, and we hope to clarify this in future work.  In fact, even in the type $A$ case, the relation between formulas in \cite{DRV13, DRV14} and identities similar to \eqref{eq:aid} involving the $P_k$ has not yet been clarified.
\end{remark}

\subsection{Explicit formulae for $n=2$}\label{sec:D_2}
In this section we write explicitly the statements above in the case $n=2$, both as an example and as a sanity check. This explicit formula will also be useful in equation \eqref{eq:deven}. 
By \cite{BB01}, when $n=2$, the section $s_2: H_2 \to BMW_2$ is given by $s_n(x) = p_1^+ x p_1^+$, where 
\[
p_1^+ := 1 - \delta^{-1}h_1.
\]
We have $P_2 = (\sigma + \sigma^{-1})/(s+s^{-1})$, and a short computation gives 
\begin{equation}\label{eq:b2}
B_2 = s_2(P_2) = \frac{\sigma + \sigma^{-1}}{s+s^{-1}} - \frac{v+v^{-1}}{s+s^{-1}}.
\end{equation}
The symmetrizer $z_2 \in H_2$ is $z_2 = (1+s\sigma)/(s^2+1)$, and another computation gives
\[
f_2 = s_2(z_2) = \frac{1+s\sigma + \beta_1 h}{s^2+1}
\]
which agrees with the $BMW_2$ symmetrizer (see, e.g. \cite{She16}). (Here $\beta_1 = -\delta^{-1}(sv^{-1}+1)$.) The power series definition of $D_2\in \cc$ shows $D_2 = 2\widehat f_2 - \widehat{1}_2$, where as above have written $\widehat f_2$ and $\widehat{1}_2$ for the closure of the symmetrizer and identity elements of $BMW_2$ in the annulus. Using relations in $BMW_2$, one can manipulate the above equations to obtain
\begin{equation}
D_2  = \widehat{B}_2 - 1
\end{equation}
which agrees with equation \eqref{eq:dvsb}.

\subsection{Recursion between symmetrizers}\label{sec:recursion}
In this section, we will recall some facts about the symmetrizers $f_n$ in the BMW algebras. 

\begin{definition}\label{def:symmetrizer}
	The \emph{symmetrizer} $f_n \in BMW_n$ is the unique idempotent such that 
	\begin{align*}
		f_n \sigma_i &= \sigma_i f_n = s \sigma_i \\
		f_n h_i &= h_i f_n = 0
	\end{align*}
	for all $i$, where $\sigma_i$ is the standard positive simple braid and $h_i$ is the cap-cup in the $i$ and $i+1$ positions. 
\end{definition}

Note that this implies that the $f_n$ are central in $BMW_n$. In the analogous case of the Hecke algebra, the symmetrizers $z_n$ are well known. In fact, each $f_n$ is the image of $z_n$ under the section of the projection map $\pi:BMW_n \to H_n$ defined in \cite{BB01}. 
In his thesis, Shelley proved the following recurrence relation between symmetrizers, which involve the constants
\[
\beta_n:=\frac{1-s^2}{s^{2n-1}v^{-1}-1}.
\]

\begin{proposition}[{\cite[Prop.\ 3]{She16}}]\label{prop:recursion}
The symmetrizers $f_n \in BMW_n$ satisfy the following recurrence relation:
	\begin{equation*}
		[n+1]f_{n+1} = [n]s^{-1}\left( f_n \otimes 1 \right) \left( 1 \otimes f_n \right) + \sigma_n \cdots \sigma_1 \left(  1 \otimes f_n \right) + [n]s^{-1}\beta_n\left( f_n \otimes 1 \right) h_n \cdots h_1 \left( 1 \otimes f_n \right).
	\end{equation*}
\end{proposition}
In terms of diagrams this identity becomes the following:
\[
[n+1] \vcenter{\hbox{\includegraphics[width=2.9cm]{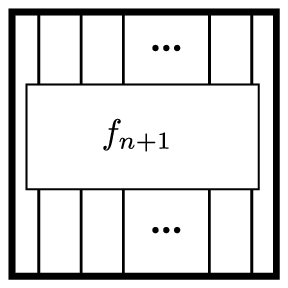}}} = [n]s^{-1} \vcenter{\hbox{\includegraphics[width=2.9cm]{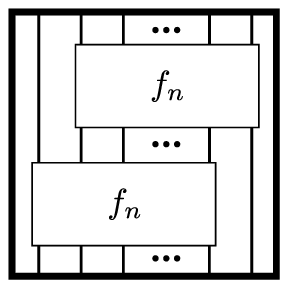}}} + \vcenter{\hbox{\includegraphics[width=2.9cm]{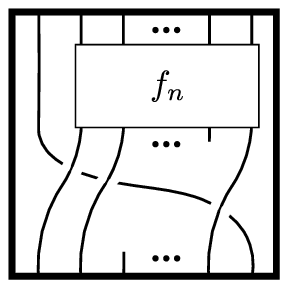}}} + [n]s^{-1} \beta_n \vcenter{\hbox{\includegraphics[width=2.9cm]{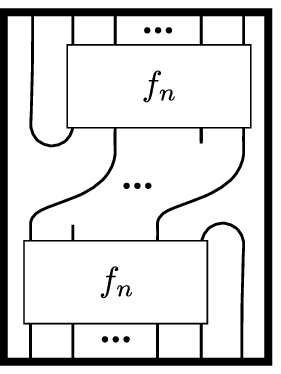}}}
\]
This recurrence relation will prove to be useful in Section \ref{sec:perpendicular}.


\HM{Does the Aiston-Morton picture of the idempotent $\lambda$ in terms of symmetrizers and anti-symmetrizers on rows and columns of the Young diagram extend to the BMW case?}
\PS{I'm also not sure here, although it seems plausible. I forget exactly how the wiring diagrams work in the Aiston-Morton picture, but the Beliakova-Blanchet section takes symmetrizers and anti-symmetrizers to the same in BMW (which makes sense, because the trivial and ``sign'' representation of the BMW algebra is the restriction of the trivial/sign reps of the Hecke algebra along the projection $BMW \twoheadrightarrow H_n$)}
\AP{I believe the answer is yes. There is a BB section $s_\lambda : H_\lambda \to BMW_\lambda$ for any object $\square_\lambda$ in the Hecke category, where $H_\lambda=\End(\square_\lambda, \square_\lambda)$. And (if I understand correctly) the Aiston-Morton idempotents $y_\lambda \in H_\lambda$ are compositions of (anti)symmetrizers tensor appropriate identity maps, so the statement follows from the fact that $s_\lambda$ is multiplicative and by Peter's comment.}
\PS{This does sound reasonable to me}

\section{All relations}
In this section we state the main theorem of the paper, and we reduce its proof to several independent propositions that we prove in later sections. The structure of this proof is similar to that of Proposition 3.7 in \cite{MS17}. Recall that the algebra $\sd$ is the Kauffman skein algebra $\sd(T^2 \times [0,1])$ of the torus, and that $D_k \in \cc = \sd(S^1\times [0,1] \times [0,1])$ is the element in the skein of the annulus defined by the power series identity \eqref{def:dk}.

\begin{definition}\label{def:D_x}
	Given $\xx = (a,b) \in \Z^2$ with $k = \gcd(a,b)$, we define $D_\xx \in \sd$ to be the image of $D_k$ under the map $\cc \to \sd$ that embeds the annulus along the simple closed curve $\xx / k$ on the torus $T^2$.
\end{definition}

\begin{remark}\label{rmk:simplecurve}
	Note that if $\xx$ is primitive (i.e. its entries have gcd 1), then $D_{\xx}$ is just the simple closed curve on the torus of slope $\xx$. The reason that the $D_\xx$ are the ``correct'' choice of generators for non-primitive $\xx$ is that they satisfy surprisingly simple relations.
\end{remark}
We first state special cases of the relations, which we prove in later sections. We then state and prove the general case of these relations.

\begin{proposition}\label{prop:specialcases}
	The elements $D_{\xx}$ satisfy the following relations:
	\begin{align}
		[D_{1,0}, D_{0,n}] &= \{n\}(D_{1,n} - D_{1,-n})\label{eq:rel1}, \\
		[D_{1,0},D_{1,n}] &= \{n\}(D_{2,n} - D_{0,n})\label{eq:rel2}.
	\end{align}
\end{proposition}
\begin{proof}
	These are proved in Sections \ref{sec:perpendicular} and \ref{sec:angled}, respectively.
\end{proof}

\begin{remark}
	Suppose we picked some other generators $D'_\xx$, and we required the $D'_\xx$ to be equivariant with respect to the $SL_2(\Z)$ action, and we choose $D'_{1,0}$ to be a simple closed curve. Then equation \eqref{eq:rel2} determines the $D'_{n,0}$ uniquely, and it can be shown that equation \eqref{eq:rel1} determines $D'_{n,0}$ up to addition of a scalar. This means the choice of generators is essentially uniquely determined by either special case of our desired relations. The key point is that the $D'_{n,0}$ are \emph{simultaneous} solutions to these two equations.
\end{remark}

Below we will show that the relations of Proposition \ref{prop:specialcases} imply the relations \eqref{eq:allrelations} desired in the main Theorem, after we prove some preparatory lemmas. We will write 
\begin{align*} 
d(\xx, \yy) &:= \det\left[\xx\,\, \yy\right] \quad \quad \,\,  \textrm{for } \xx, \yy \in \Z^2, \\
d(\xx) &:= gcd(m,n) \quad \quad \textrm{when } \xx = (m,n).
\end{align*} 
We will also use the following terminology: 
\[
(\xx,\yy) \in \Z \times \Z \textrm{ is \emph{good} if  } [D_\xx,D_\yy] = \{d(\xx,\yy)\} \left( D_{\xx+\yy}-D_{\xx-\yy} \right).
\]

\begin{remark}\label{remark_goodsymmetry}
Note that because $D_{\xx}=D_{-\xx}$, if $(\xx, \yy)$ is good, then the pairs $(\pm\xx, \pm\yy)$ are good as well. 
\end{remark}

The idea of the proof is to induct on the absolute value of the determinant of the matrix with columns $\xx$ and $\yy$. To induct, we write $\xx = \aab + \bb$ for carefully chosen vectors $\aab, \bb$ and then use the following lemma. 

\begin{lemma}\label{lemma_trueforab}
	Assume $\aab + \bb = \xx$ and that $(\aab,\bb)$ is good. Further assume that the five pairs of vectors $(\yy, \aab)$, $(\yy, \bb)$, $(\yy+\aab,\bb)$, $(\yy+\bb,\aab)$, and $(\aab-\bb, \yy)$, are good. Then the pair $(\xx,\yy)$ is good.
\end{lemma}
\begin{proof}
We  use the Jacobi identity and the assumptions to compute
	\begin{eqnarray*}
&-&\{d(\aab, \bb)\} [D_{\aab+\bb}, D_y] + \{d(\aab, \bb)\} [D_{\aab-\bb}, D_y] \\
=&-&[[D_{\aab}, D_{\bb}], D_{\yy}] \\
=&&[[D_{\yy}, D_{\aab}], D_{\bb}] + [[D_{\bb}, D_y], D_{\aab}] \\
=&&\{d(\yy, \aab)\} [D_{\yy+\aab} - D_{\yy-\aab}, D_{\bb}] + \{d(b,y)\} [D_{\bb+\yy} - D_{\bb-\yy}], D_{\aab}] \\
=&&\{d(\yy,\aab)\} \left( \{d(\yy+\aab, \bb)\} \left( D_{\yy+\aab+\bb} - D_{\yy+\aab-\bb} \right) \right. \\
&& \left. \qquad\qquad -\{d(\yy-\aab,\bb)\} \left( D_{\yy-\aab+\bb} - D_{\yy-\aab-\bb} \right) \right) \\
&+&\{d(\bb,\yy)\} \left( \{d(\bb+\yy, \aab)\} \left( D_{\bb+\yy+\aab} - D_{\bb+\yy-\aab} \right) \right. \\
&& \left. \qquad\qquad -\{d(\bb-\yy, \aab)\} \left( D_{\bb-\yy+\aab} - D_{\bb-\yy-\aab} \right) \right) \\
=&& \left( \{d(\yy,\aab)\} \{d(\yy+\aab,\bb)\} + \{d(\bb,\yy)\} \{d(\bb+\yy,\aab)\} \right) D_{\aab+\bb+\yy} \\
&+& \left( \{d(\yy,\aab)\} \{d(\yy-\aab,\bb)\} - \{d(\bb,\yy)\} \{d(\bb-\yy,\aab)\} \right) D_{\aab+\bb-\yy} \\
&-& \left( \{d(\yy,\aab)\} \{d(\yy+\aab,\bb)\} - \{d(\bb,\yy)\} \{d(\bb-\yy,\aab)\} \right) D_{\aab-\bb+\yy} \\
&-& \left( \{d(\yy,\aab)\} \{d(\yy-\aab,\bb)\} + \{d(\bb,\yy)\} \{d(\bb+\yy,\aab)\} \right) D_{\aab-\bb-\yy} \\
=:&& c_1 D_{\aab+\bb+\yy} + c_2 D_{\aab+\bb-\yy} - c_3 D_{\aab-\bb+\yy} - c_4 D_{\aab-\bb-\yy}.
	\end{eqnarray*}
Using some simple algebra, we can show
	\begin{eqnarray*}
c_1 \,\,\, =& &  \{d(\yy,\aab)\} \{d(\yy+\aab,\bb)\} + \{d(\bb,\yy)\} \{d(\bb+\yy,\aab)\} \\
=&& \{d(\yy,\aab)+d(\yy+\aab,\bb)\}^+ - \{d(\yy,\aab)-d(\yy+\aab,\bb)\}^+ \\
&+& \{d(\bb,\yy)+d(\bb+\yy,\aab)\}^+ - \{d(\bb,\yy)-d(\bb+\yy,\aab)\}^+  \\
=&& \{d(\yy,\aab+\bb)+d(\aab,\bb)\}^+ - \{d(\yy,\aab-\bb)-d(\aab,\bb)\}^+ \\
&+& \{d(\yy,\aab-\bb)-d(\aab,\bb)\}^+ - \{d(\aab,\bb)-d(\yy,\aab+\bb)\}^+ \\
=&& \{d(\aab,\bb)+d(\yy,\aab+\bb)\}^+ - \{d(\aab,\bb)-d(\yy,\aab+\bb)\}^+ \\
=&& \{d(\aab,\bb)\} \{d(\yy,\aab+\bb)\} \\
=&& - \{d(\aab,\bb)\} \{d(\xx,\yy)\}.
	\end{eqnarray*}
Similar computations for the other $c_i$ show that 
	\begin{align}
\frac{-1}{\{d(\aab,\bb)\}}[[D_\aab,D_\bb],D_\yy] &=  \frac{1}{\{d(\aab,\bb)\}} \left( c_1 D_{\aab+\bb+\yy} + c_2 D_{\aab+\bb-\yy} - c_3 D_{\aab-\bb+\yy} - c_4 D_{\aab-\bb-\yy}\right) \notag\\ 
&=  -\{d(\xx,\yy)\} \left( D_{\xx+\yy} - D_{\xx-\yy} \right) 
+\{d(\aab-\bb,\yy)\} \left( D_{\aab-\bb+\yy} - D_{\aab-\bb-\yy} \right)\notag \\
&= -\{d(\xx,\yy)\} \left( D_{\xx+\yy} - D_{\xx-\yy} \right) 
+ [D_{\aab-\bb},D_\yy]. \label{eq:astep}
	\end{align}
	By the assumption that $(\aab,\bb)$ is good, we have 
\begin{equation}\label{eq:added}
[D_{\aab}, D_{\bb}] = \{d(\aab, \bb )\} \big( D_{\xx} - D_{\aab-\bb} \big).
\end{equation}
Finally, combining equations \eqref{eq:added} and \eqref{eq:astep} shows that the pair $(\xx,\yy)$ is good.
\end{proof}

We next prove the following elementary lemma (which is a slight modification of \cite[Lemma 1]{FG00}). This lemma is used to make a careful choice of vectors $\aab, \bb$ so that the previous lemma can be applied. 

\begin{lemma}\label{lemma_diophantine}
	Suppose $p,q \in \Z$ are relatively prime with $0 < q < p$ and $p > 1$. Then there exist $u,v,w,z \in \Z$ such that the following conditions hold:
	\begin{eqnarray}
	u + w &=& p,\quad v + z = q \notag \\
	0 < u, w &<& p\label{equation_conditionsonuvwz} \\
	uz - wv &=& 1. \notag
	\end{eqnarray}
\end{lemma}
\begin{proof}
	Since $p$ and $q$ are relatively prime, there exist $a,b \in \Z$ with $ bq - ap = 1$. This solution can be modified to give another solution $a' = a + q$ and $b' = b + p$, so we may assume $0 \leq b < p$. We then define 
	\[
	u=b,\quad v=a,\quad w=p-b,\quad z=q-a.
	\]
	By definition, $u,v,w,z$ satisfy the first condition of \eqref{equation_conditionsonuvwz}, and the inequalities $0 \leq b < p$ and $p > 1$ imply the second condition. To finish the proof, we compute
	\[
	uz - wv = b(q-a) - a(p-b) = bq - ap = 1.
	\]
\end{proof}

It might be helpful to point out that the numbers in the previous lemma satisfy
\[
\left[ \begin{array}{cc} u&w\\v&z\end{array}\right]\left[\begin{array}{c}1\\1\end{array}\right] = \left[ \begin{array}{c}p\\q\end{array}\right].
\]

\begin{remark}\label{remark_gl2z}
	There is a natural $R$-linear anti-automorphism $\tau:\sd \to \sd$ which ``flips $T^2 \times [0,1]$ across the $y$-axis and inverts $[0,1]$.'' In terms of the elements $D_{a,b}$, we have $\tau(D_{a,b}) = D_{a,-b}$. We therefore have an \emph{a priori} action of $\mathrm{GL}_2(\Z)$ on $\sd$, where elements of determinant $1$ act by algebra automorphisms, and elements of determinant $-1$ act by algebra anti-automorphisms. It will be important that the orbits of this action on the $D_{\xx}$ are the fibers of the assignment $D_{\xx} \mapsto d(\xx)$ (which is essentially the statement of Lemma \ref{lemma_diophantine}). In other words, up to the action of $GL_2(\mathbb Z)$, vectors in $\Z^2$ are classified by the GCD of their entries.
\end{remark}

\begin{proposition}\label{lemma_allfromsome}
	Suppose $A$ is an algebra with elements $D_\xx$ for $\xx \in \Z^2/\langle -\xx = \xx \rangle$ that satisfy equations \eqref{eq:rel1} and \eqref{eq:rel2}. Furthermore, suppose that there is a $\GL_2(\Z)$ action by (anti-)automorphisms on $A$ as in Remark \ref{remark_gl2z}, and that the action of $\GL_2(\Z)$ is given by $\gamma(D_\xx) = D_{\gamma(\xx)}$  for $\gamma \in \GL_2(\Z)$.  Then the $D_\xx$ satisfy the equations 
	\begin{equation*}
		[D_\xx, D_\yy] = (s^d - s^{-d})\left( D_{\xx+\yy} - D_{\xx-\yy}\right).
	\end{equation*}
\end{proposition}
\begin{proof}
	The proof proceeds by induction on $\lvert d(\xx,\yy)\rvert $, and the base case $\lvert d(\xx,\yy)\rvert = 1$ is immediate from Remark \ref{remark_gl2z} and the assumption \eqref{eq:rel1} for $\xx = (1,0)$ and $\yy = (0,1)$. We now make the following inductive assumption:
	\begin{equation}\label{assumption1} 
	\textrm{For all } \xx',\yy' \in \Z^2\textrm{ with } \lvert d(\xx',\yy')\rvert < \lvert d(\xx,\yy) \rvert, \textrm{ the pair } (\xx',\yy') \textrm{ is good.}
	\end{equation}
	We would like to show that $[D_{\xx},D_{\yy}] = \{d(\xx,\yy)\} \big( D_{\xx + \yy} - D_{\xx-\yy} \big)$. By Remark \ref{remark_gl2z}, we may assume
	\[
	\yy = \begin{pmatrix}0\\r\end{pmatrix},\quad \xx = \begin{pmatrix}p\\q\end{pmatrix},\quad d(\xx) \leq d(\yy),\quad 0 \leq q < p.
	\]
	
	If $p=1$, then this equation follows from \eqref{eq:rel2}, so we may also assume $p > 1$. Furthermore, we may assume that $r>0$ by Remark \ref{remark_goodsymmetry}. 

	We will now show that if either $d(\xx)=1$ or $d(\yy)=1$, then $(\xx,\yy)$ is good. By symmetry of the above construction of $\xx$ and $\yy$, we may assume $d(\xx)=1$, which immediately implies $q>0$. Furthermore, we may now assume that $r>1$ by the relation \eqref{eq:rel2}. We apply Lemma \ref{lemma_diophantine} to $p, q$ to obtain  $u,v,w,z \in \Z$ satisfying 
	\begin{equation}\label{assumption1.49}
	uz - vw = 1,\quad uq - vp = 1,\quad u + w = p,\quad v + z = q,\quad 0 < u,w < p.
	\end{equation}
	We then define vectors $\aab$ and $\bb$ as follows:
	\begin{equation}\label{assumption1.9}
	\aab := \begin{pmatrix}  u\\ v\end{pmatrix},\quad 
	\bb := \begin{pmatrix} w\\ z\end{pmatrix},
	\quad \aab + \bb = \xx,\quad d(\aab, \bb) = 1. 
	\end{equation}

	Using Lemma \ref{lemma_trueforab} and Assumption (\ref{assumption1}), it is sufficient to show that each of $\lvert d(\aab, \bb) \rvert$, $\lvert d(\yy,\bb) \rvert$, $ \lvert d(\yy,\aab) \rvert$, $\lvert d(\yy+\aab,\bb) \rvert$, $\lvert d(\yy+\bb,\aab) \rvert$, and $\lvert d(\aab-\bb,\yy) \rvert$ are strictly less than $pr = \lvert d(\xx,\yy) \rvert$. First, $ \lvert d(\aab,\bb) \rvert = 1$ is strictly less than $pr$ since $p>1$ and $r>0$. Second, $ \lvert d(\yy,\aab) \rvert = ur $ and $ \lvert d(\yy,\bb) \rvert = wr $ are strictly less than $pr$ by the inequalities in (\ref{assumption1.49}). Third, we compute
	\begin{eqnarray*}
		\vert d(\yy+\aab, \bb) \rvert &=&\vert - d(\yy+\aab, \bb) \rvert \\
		&=& \lvert - d(\yy,\bb) - d(\aab, \bb) \rvert \\
		&=& \lvert wr - 1 \rvert \\
		&=& wr-1 \\
		&<& wr \\
		&<& pr.
	\end{eqnarray*}

	Fourth, we compute
	\begin{eqnarray*}
		\lvert -d(\yy+\bb, \aab) \rvert &=& \lvert -d(\yy+\bb, \aab) \rvert \\
		&=& \lvert -d(\yy,\aab) - d(\bb, \aab) \rvert \\
		&=& \lvert ur + 1 \rvert \\
		&=& ur+1 \\
		&\leq& (p-1)r+1 \\
		&=& pr-r+1 \\
		&<& pr.
	\end{eqnarray*}

	Finally, we compute
	\begin{eqnarray*}
		\lvert d(\aab-\bb,\yy) \rvert &=& \lvert d(\aab,\yy) - d(\bb,\yy) \rvert \\
		&=& \lvert d(\yy,\aab) - d(\yy,\bb) \rvert \\
		&=& \lvert ur - wr \rvert \\
		&=& \lvert u-w \rvert r \\
		&<& \lvert u+w \rvert r \\
		&=& pr.	
	\end{eqnarray*}
	
	So we have shown that $(\xx,\yy)$ is good if $d(\xx)=1$ or $d(\yy)=1$. Let us now turn our attention to the more general case. We will immediately split this into cases depending on $q$.
	
	\noindent \emph{Case 1:} Assume $0 < q$. 
	
	Let $p' = p / d(\xx)$ and $q' = q / d(\xx)$. By the assumption $0 < q$, we see that $d(\xx) < p$, so $p' > 1$. We can therefore apply Lemma \ref{lemma_diophantine} to $p',q'$ to obtain $u,v,w,z \in \Z$ satisfying 
	\begin{equation}\label{assumption1.5}
	uz - vw = 1,\quad uq' - vp' = 1,\quad u + w = p',\quad v + z = q',\quad 0 < u,w < p'.
	\end{equation}
	In a way similar to the above, we may pick vectors $\aab$ and $\bb$ as follows: 
	\begin{equation}\label{assumption2}
	\aab := \begin{pmatrix}  d(\xx)u\\ d(\xx)v\end{pmatrix},\quad 
	\bb := \begin{pmatrix} d(\xx)w\\ d(\xx)z\end{pmatrix},
	\quad \aab + \bb = \xx,\quad d(\aab, \bb) = d(\xx)^2. 
	\end{equation}
	
	As before, it is sufficient to show that each of $\lvert d(\aab, \bb) \rvert$, $\lvert d(\yy,\bb) \rvert$, $\lvert d(\yy,\aab) \rvert$, $\lvert d(\yy+\aab,\bb) \rvert$, $\lvert d(\yy+\bb,\aab) \rvert$, and $\lvert d(\aab-\bb,\yy) \rvert$ are strictly less than $pr  = \lvert d(\xx,\yy) \rvert$. First,
	\[
	\lvert d(\aab,\bb) \rvert = d(\xx)^2 \leq d(\xx)d(\yy) = d(\xx)r < pr
	\]
	where the last inequality follows from the assumption $0 < q < p$. Second, we can compute $\lvert d(\yy,\bb) \rvert = d(\xx)w r $ and $ \lvert d(\yy,\aab) \rvert =  d(\xx)u r $ are strictly less than $pr$ by the inequalities in (\ref{assumption1.5}). Third, we compute
	\begin{eqnarray*}
		\lvert d(\yy+\aab, \bb) \rvert &=& \lvert -d(\yy+\aab, \bb) \rvert \\
		&=& \lvert -d(\yy,\bb) - d(\aab, \bb) \rvert\\
		&=& \lvert d(\xx)wr - d(\xx)^2 \rvert \\
		&=& d(\xx)wr - d(\xx)^2 \\ 
		&<& d(\xx)wr\\
		&\leq& pr.
	\end{eqnarray*}
	Finally, we compute
	\begin{eqnarray*}
		\lvert d(\yy+\bb, \aab) \rvert &=& \lvert -d(\yy+\bb, \aab) \rvert \\
		&=& \lvert -d(\yy,\aab) - d(\bb, \aab) \rvert \\
		&=& d(\xx)ur + d(\xx)^2 \\
		&\leq& \left(d(\xx)u + d(\xx)\right)d(\yy)\\
		&=& (u+1)d(\xx)r.
	\end{eqnarray*}
	
	Therefore, we will be finished once we show that $(u+1)d(\xx)$ is strictly less than $p$. We now split into subcases:\\[2mm]
	
	\noindent\emph{Subcase 1a:} If $u + 1 < p'$, then $(u+1)d(\xx)r < p'd(\xx)r = pr$, and we are done. 
	
	\noindent \emph{Subcase 1b:} Assume $u + 1 = p'$. By equation (\ref{assumption1.5}), we have 
	\[
	1 = uq' - vp' = (p'-1)q' - vp'  \implies p'(q'-v) = 1 + q' < 1 + p'.
	\]
	Since $p' > 1$, the last inequality implies $q' - v = 1$, which implies $v = q'-1$ and $z = 1$. Since $uz-vw = 1$, this implies $(p'-1)-(q'-1) = 1$, which implies $q' = p'-1$. If we write $g = d(\xx)$, we then have
	\[
	\lvert d(\yy + \bb, \aab) \rvert = \lvert -d(\yy + \bb, \aab) \rvert = \Bigg| -\det\left[ \begin{array}{cc}g&p -g\\g + r&p-2g\end{array}\right] \Bigg| =  \lvert rp + g(g-r) \rvert = rp + d(\xx)(d(\xx)-r) \leq rp
	\]
	where the last inequality comes from the assumption $d(\xx) \leq r$. If this inequality is strict, then we are done. Otherwise, we move onto the next subcase.
	
	\noindent\emph{Subcase 1c:} In this subcase, we are reduced to showing the following pair of vectors is good: 
	\[
	\yy = (0,r),\quad \quad \xx = (rp', rp'-r).
	\]
	If $r=1$, then $d(\yy)=1$, which makes $(\xx,\yy)$ good. Thus, we may assume that $r>1$. 
	
	We must replace our previous choice of $\aab$ and $\bb$ with a choice which is better adapted to this particular subcase. We define
	\[
	\aab := \begin{pmatrix} 1\\-1 \end{pmatrix},\quad \bb := \begin{pmatrix} rp'-1\\rp' - r + 1 \end{pmatrix}.
	\]

	We know that the pairs $(\aab, \bb), (\yy, \aab), (\yy+\bb,\aab)$ are good since $d(\aab)=1$. Since $r > 1$ and $p'>1$, we can compute that the determinants of the pairs $(\yy,\bb), (\yy+\aab,\bb), (\aab-\bb,\yy)$ are strictly less than $r^2p' = \lvert d(\xx,\yy) \rvert$ as follows:
	\begin{eqnarray*}
		\lvert d(\aab-\bb,\yy)\rvert &=&  \lvert r^2p' - 2r \rvert \\
		&=& r^2p' - 2r \\
		&<& r^2p'
	\end{eqnarray*}

	\begin{eqnarray*}
		\lvert d(\yy,\bb)\rvert &=&  \lvert r^2p' - r \rvert \\
		&=& r^2p' - r \\
		&<& r^2p'
	\end{eqnarray*}

	\begin{eqnarray*}
		\lvert d(\yy+\aab, \bb)\rvert &=&  \lvert r^2p' - (rp'+r-1) \rvert \\
		&=&  r^2p' - (rp'+r-1) \\
		&<& r^2p'.
	\end{eqnarray*}

This together with Assumption (\ref{assumption1}) and Lemma \ref{lemma_trueforab} shows that $(\xx,\yy)$ is good, which finishes the proof of this subcase and finishes the proof of Case 1. \\[2mm]
	
	\noindent \emph{Case 2:} In this case we assume $q=0$. We define $\aab, \bb$ similarly to Subcase 1c, so we have
	\[
	\yy = \begin{pmatrix} 0\\r\end{pmatrix},\quad \xx = \begin{pmatrix} p\\0\end{pmatrix},\quad \aab := \begin{pmatrix} 1\\-1 \end{pmatrix},\quad \bb := \begin{pmatrix} p-1\\ 1 \end{pmatrix}.
	\]

	Since $d(\aab) = d(\bb) = 1$, the pairs $(\aab,\bb), (\yy,\aab), (\yy,\bb), (\yy+\aab,\bb), (\yy+\bb,\aab)$ are all good. We must check that $\lvert d(\aab-\bb,\yy) \rvert < pr = \lvert d(\xx,\yy) \rvert$. If $r=1$, then the relation \eqref{eq:rel1} implies that the pair $(\xx,\yy)$ is good. Thus, we may assume that $r>1$. We may also assume that $p>1$. Finally, we check $\lvert d(\aab-\bb,\yy) \rvert = \lvert rp-2r \rvert = rp-2r < rp$. By using Assumption (\ref{assumption1}) and Lemma \ref{lemma_trueforab}, this completes Case 2 which completes the proof. 
\end{proof}

	Using the technical results above, we give a presentation of the algebra $\sd$.
\begin{theorem}\label{thm:presentation}
	The algebra $\sd$ is generated by the elements $D_\xx$, and these satisfy the following relations:
\begin{align}\label{eq:allrelations}
	[D_\xx,D_\yy] &= (s^d-s^{-d})(D_{\xx+\yy}-D_{\xx-\yy})\\
	D_\xx &= D_{-\xx} \notag
\end{align}
	where $d = \det(\xx \, \yy)$. This gives a presentation of $\sd$ as an algebra.
\end{theorem}
\begin{remark}
	In the simple case $\xx=(1,0)$ and $\yy = (0,1)$, the claimed relation is exactly the Kauffman skein relation. Also, note that if we take the first equation and replace $\xx$ with $-\xx$, the left hand side is invariant because of the second relation, and the right hand side is invariant because both factors switch sign.
\end{remark}
\begin{proof}
	First we will show that $\sd$ is generated by the $D_{\xx}$. Using the skein relation, we may write an arbitrary element of $\sd$ as a sum of products of knots (i.e.  1-component links). Using further skein relations, any knot can be written as a sum of products of annular knots (i.e. knots which are contained in an annulus inside the torus). \PS{Are we using the following statement, and isit true? ``Any ascending diagram of a knot in $T^2$ is isotopic to one contained in an embedded annulus.''} Therefore, it is sufficient to prove that the $D_k$ generate the skein algebra of the annulus. 
	
	By \cite[Cor. 2]{LZ02}, the closures $Q_\lambda$ are a basis of the skein of the annulus, so we need to show each $Q_\lambda$ can be written as a polynomial in the $D_k$. We prove this by induction in $n = \lvert \lambda \rvert$, and the base case is trivial. By  \cite[Lemma 3.1]{MS17}, the closures of the $P_k$ generate the Homflypt skein algebra of the annulus. In particular, given any Young diagram $\lambda$ with $n$ boxes, we can find a polynomial $g_\lambda(x_1,x_2,\cdots, x_n)$ so that $g_\lambda(P_1,\cdots,P_n) = s_{\lambda,\varnothing}$. Lemma \ref{lemma:inc} then shows that $g_\lambda(B_1,\cdots,B_n) = Q_\lambda + cl(a)$, where $a \in BMW_n$ is in the kernel of the projection to the Hecke algebra $H_n$. This kernel is generated by the ``cup-cap'' elements, which implies the closure $cl(a)$ is equal to the closure of an element of $BMW_{n-2}$. The proof of generation is completed by noting that the $B_k$ and $D_k$ generate the same algebra.
\PS{I rewrote this paragraph after a comment by Alex}

	It is clear that $D_{\xx}=D_{-\xx}$ due to the lack of orientation on the links. The other relations have already been shown to hold in Proposition \ref{lemma_allfromsome}. Finally, to show that these relations give a presentation of $\sd$, we first note that Theorem \ref{thm:basis} shows that as a vector space, $\sd$ has a basis given by unordered words in the elements $D_\xx$. The relations  \eqref{eq:allrelations} allow any $D_\xx$ and $D_\yy$ to be reordered, which shows that they give a presentation of $\sd$.
\end{proof}

\section{Perpendicular relations}\label{sec:perpendicular}
In this section we prove the perpendicular relations 
\begin{equation}\label{eq:perp}
	[D_{1,0}, D_{0,n}] = \{n\}(D_{1,n} - D_{1,-n}).
\end{equation}
These relations are between elements which are contained in annuli which intersect in just one disk, and the ``angled relations'' we prove in the next section involve elements which are contained in annuli that intersect in several points.

\HM{It would be nice to say what we have in mind for the terms 'perpendicular' and 'angled' in these two sections.  Is it helpful to note that we are dealing with two annuli in the torus whose cores meet in just 1 point, while in the angled case they meet in n>1 points?} \PS{I added the sentence below the equation about this}

The main tool we use will be the recursion relation between the symmetrizers in the BMW algebra from Section \ref{sec:recursion}. We will actually show a stronger identity, which is the analogue of an identity involving the Hecke algebra elements $P_k$ found by Morton and coauthors.


Let $\sa$ be the skein algebra of the annulus with one marked point on each boundary component. For $x \in \cc$, the element $l(x) \in \sa$ is defined to be a horizontal arc joining the boundary points with the element $x$ placed below the arc. Similarly, we may define $r(x)$, but the element $x$ is placed above the arc. 
\[
	l(x) = \vcenter{\hbox{\includegraphics[width=5cm]{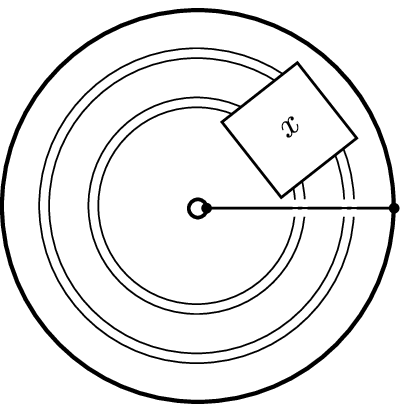}}} \qquad\qquad\qquad r(x) = \vcenter{\hbox{\includegraphics[width=5cm]{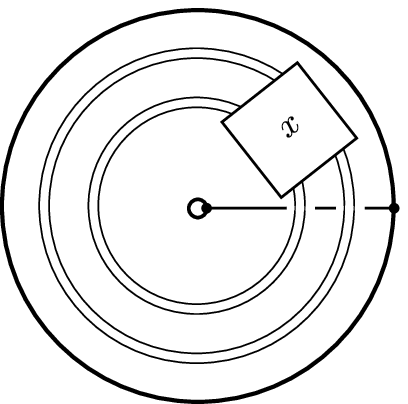}}}
\]

Let $a$ be the arc that goes once around the annulus in the counter-clockwise direction: 
\[
	 a=\vcenter{\hbox{\includegraphics[width=5cm]{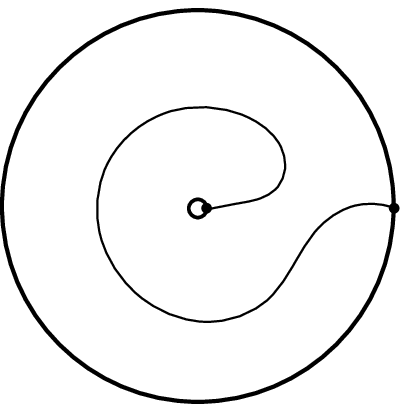}}}.
\]
The vector space $\sa$ is an algebra, where the multiplication is given by ``nesting'' annuli. The identity in this algebra is $l(\varnothing)$, where $\varnothing \in \cc$ is the empty link. In fact $\sa$ is commutative, and is isomorphic to $\cc \otimes_R R[a^{\pm 1}]$ (see \cite{She16}), but we won't need this fact. We will, however, use the fact that $l(-)$ and $r(-)$ are algebra maps from $\cc$ to $\sa$.

\begin{theorem}\label{thm:aid}
	We have the following identity in $\sa$:
\begin{equation}\label{eq:aid}
	l(D_n) - r(D_n) = \{n\}(a^n - a^{-n}).
\end{equation}
\end{theorem}
Note that this relation holds for the $B_k$ as well, as stated in Section \ref{sec:intro}.
Also, note that this implies \eqref{eq:perp}.
In particular, by Remark \ref{rmk:simplecurve}, after applying the wiring $cl_1:\sa\to\sd$ which closes the marked points around the torus as shown:
\[
	\vcenter{\hbox{\includegraphics[width=5cm]{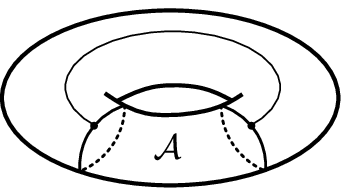}}}
\]
 equation \eqref{eq:aid} becomes  
\[
	D_{1,0}D_{0,n} - D_{0,n}D_{1,0} = \{n\}(D_{1,n} - D_{1,-n}).
\]

\HM{I like this result. Coule we refer directly to the Homflypt counterpart here?} \PS{I added the following comment, which I agree is a good idea, and the reference from your email.}
\begin{remark}
	The identity \eqref{eq:aid} is the $BMW$ analogue of an identity proved by Morton in \cite[Thm. 4.2]{Mor02b}. In that case, the identity reads $l(P_n) - r(P_n) = \{n\}a^n$, and in both settings the $n=1$ case follows immediately from the skein relation. However, it is not clear whether the results of Beliakova and Blanchet in \cite{BB01} can be used to prove \eqref{eq:aid} directly from the Homflypt identity.
\end{remark}

The rest of the section is dedicated to the proof of Theorem \ref{thm:aid}, and we begin with some lemmas. In an effort to reduce notation, define the following elements of $\sa$:
\begin{align*}
c&:=l(D_n) - r(D_n) - \{n\}(a^n - a^{-n}), \\
l_n &:= l \left( \widehat{f}_n \right), \\
r_n &:= r \left( \widehat{f}_n \right).
\end{align*} 
As a convention, let $l_0=r_0=l_{-1}=r_{-1}=1_{\sa}$. We would like to show $c=0$.
	
\begin{lemma}\label{lemma:perprel}
	The relation $c=0$ in $\sa$ is equivalent to the following:
\begin{align}
	l_{n+2} -  r_{n+2} &= (sa + s^{-1}a^{-1}) l_{n+1} - (s^{-1}a + sa^{-1}) r_{n+1} - (l_{n} - r_{n}); \quad  n \geq -1. \label{eq:perprel3}
\end{align}
\end{lemma}
\begin{proof}
	We may rephrase the relation $c=0$ as the following power series identity in $\sa[[t]]$.
\begin{equation*}
	\sum_{n=1}^{\infty}  \frac{l(D_n) - r(D_ n)}{n}t^n = \sum_{n=1}^{\infty} \frac{\{n\}(a^n-a^{-n})}{n}t^n.
\end{equation*}
	Using linearity of $l$ and $r$ on the left hand side, while expanding out the right hand side we get the equation:
\begin{align*}
	l\left( \sum_{n=1}^{\infty} \frac{D_n}{n}t^n \right) - r\left( \sum_{n=1}^{\infty} \frac{D_n}{n}t^n \right) &= \sum_{n=1}^{\infty} \frac{s^{n}a^{n}t^n}{n} + \sum_{n=1}^{\infty} \frac{s^{-n}a^{-n}t^n}{n} - \sum_{n=1}^{\infty} \frac{s^{n}a^{-n}t^n}{n} - \sum_{n=1}^{\infty} \frac{s^{-n}a^{n}t^n}{n}.
\end{align*}

Define the power series
\[
	F(t) := 1+\sum_{n=1}^{\infty} \widehat{f}_n t^n \in \cc[[t]].
\]
	Then on the left hand side we may use equation \eqref{eq:dk} and on the right side we may apply Newton's power series identity to obtain:
\begin{equation*}
	l\left( \mathrm{ln}\left( F(t) \right) \right) - r\left( \mathrm{ln}\left( F(t) \right) \right) = -\mathrm{ln}(1-sat) - \mathrm{ln}(1-s^{-1}a^{-1}t) + \mathrm{ln}(1-sa^{-1}t) + \mathrm{ln}(1-s^{-1}at).
\end{equation*}
	Moving terms around and using properties of natural log, we arrive at the equation:
\begin{equation*}
	\mathrm{ln}( l(F(t)) (1-(sa + s^{-1}a^{-1})t +t^2)) = \mathrm{ln}( r( F(t)) (1 - (sa^{-1} + s^{-1}a)t+t^2)).
\end{equation*}
	Exponentiating both sides, we get:
\begin{equation*}
	l(F(t)) (1-(sa + s^{-1}a^{-1})t +t^2) = r( F(t)) (1 - (sa^{-1} + s^{-1}a)t+t^2).
\end{equation*}
	Equate the coefficients of $t$ to obtain a system of equations in $\sa$:
\begin{align*}
	l_1 - (sa + s^{-1}a^{-1}) &= r_1 - (sa^{-1} + s^{-1}a)\\
	l_2 - (sa + s^{-1}a^{-1})l_1 +1 &= r_2 - (sa^{-1}+s^{-1}a)r_1 +1\\
	l_{n} - (sa + s^{-1}a^{-1})l_{n-1} + l_{n-2} &= r_{n} - (sa^{-1} + s^{-1}a)r_{n-1} + r_{n-2}; \quad n \geq 3.
\end{align*}
	These terms can be reindexed to obtain \eqref{eq:perprel3}, thus completing the proof. 
\end{proof}

Before showing the relation \eqref{eq:perprel3}, we will first set up some machinery in $\sa$. First, consider the following wirings $W_n, \widetilde{W}_n:BMW_{n+1}\to\sa$ defined by: 
\[
	W_n = \vcenter{\hbox{\includegraphics[width=5cm]{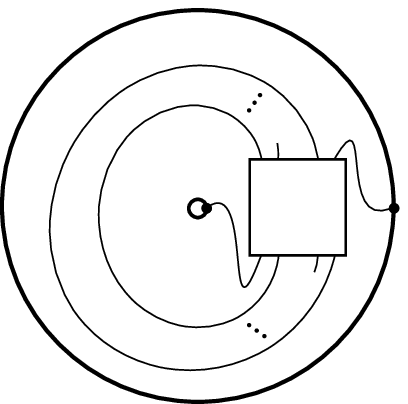}}}, \qquad\qquad\qquad \widetilde{W}_n = \vcenter{\hbox{\includegraphics[width=5cm]{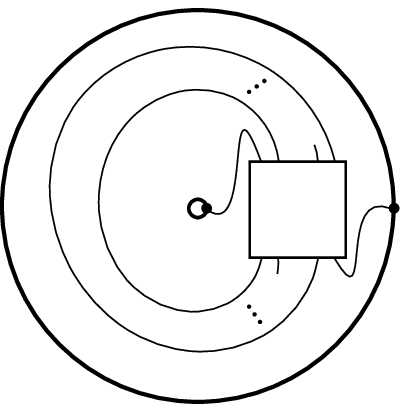}}}.
\]

At the risk of abusing notation, we will set $W_n := W_n \left( \widehat{f}_{n+1}\right)$ and $\widetilde{W}_n := \widetilde{W}_n \left( \widehat{f}_{n+1}\right)$. We can see the following identities diagrammatically: 
\begin{eqnarray}
	W_n \left( x \otimes 1 \right) = W_n \left( 1 \otimes x \right) = a W_n \left( x \right) \\
	\widetilde{W}_n \left( x \otimes 1 \right) = \widetilde{W}_n \left( 1 \otimes x \right) = a^{-1} \widetilde{W}_n \left( x \right).
\end{eqnarray}

We will also need the constants\\ 
\[
	\beta_n:=\frac{1-s^2}{s^{2n-1}v^{-1}-1}.
\]

\begin{remark}\label{rmk:involutions}
There exist maps
\[
	\tau: BMW_n \to BMW_n \qquad \overline{\,\cdot\,}: BMW_n \to BMW_n
\]
induced by the diffeomorphisms of the thickened square $(x, y, t) \mapsto (x, 1-y, 1-t)$ and $(x, y, t) \mapsto (x, y, 1-t)$, respectively. The map $\overline{\,\cdot\,}$ is often called the \emph{mirror map} and is an $R$-anti-linear involution, while $\tau$ will be called the \emph{flip map} and is an $R$-linear anti-involution which extends the dihedral symmetry of the square to the thickened square. As noted in \cite{She16}, the symmetrizers $f_n$ are fixed under these maps.
Using the quotient map defined by the equivalence relation $(x, 0, t) \sim (x, 1, t)$, we may analagously define maps
\[
	\tau: \sa \to \sa, \qquad \overline{\,\cdot\,}: \sa \to \sa
\]
which are linear and anti-linear involutions, respectively. We will also call these the flip map and the mirror map; it will be clear from the context which is being applied. These maps satisfy:
\begin{eqnarray}
	\tau \left( W_n \left( x \right) \right) &=& \widetilde{W}_n \left( \tau \left( x \right) \right) \label{tauW} \\
	\overline{W_n \left( x \right)} &=& W_n \left( \overline{x} \right).
\end{eqnarray}

\end{remark}

\begin{lemma}\label{lemma:recursionina}
	The following relations hold in $\sa$. 
	\begin{align}
		l_n &= [n+1]W_n - [n]s^{-1}aW_{n-1} - [n]s^{-1}\beta_na^{-1}\widetilde{W}_{n-1} \label{eq:recursionina1} \\
		l_n &= [n+1]\widetilde{W}_n - [n]sa^{-1}\widetilde{W}_{n-1} - [n]s\bar{\beta}_naW_{n-1} \label{eq:recursionina2} \\
		r_n &= [n+1]W_n - [n]saW_{n-1} - [n]s\bar{\beta}_na^{-1}\widetilde{W}_{n-1} \label{eq:recursionina3} \\
		r_n &= [n+1]\widetilde{W}_n - [n]s^{-1}a^{-1}\widetilde{W}_{n-1} - [n]s^{-1}\beta_naW_{n-1} \label{eq:recursionina4} 
	\end{align}
\end{lemma}
\begin{proof}
This was observed in [She16] using different notation. An outline of the proof is as follows. 
First, notice that the equations \eqref{eq:recursionina1} and \eqref{eq:recursionina2} are the result of applying the mirror map to equations \eqref{eq:recursionina3} and \eqref{eq:recursionina4} respectively. By \eqref{tauW}, one can quickly verify that equation \eqref{eq:recursionina1} is the image of equation \eqref{eq:recursionina4} under the flip map. In a similar way, \eqref{eq:recursionina1} and \eqref{eq:recursionina2} are the images of \eqref{eq:recursionina4} and \eqref{eq:recursionina3} respectively under the flip map $\tau$. Thus, it suffices to prove equation \eqref{eq:recursionina4}. To do this, we will use Proposition 3 from \cite{She16} which provides a recurrence relation of the symmetrizers of $BMW_{n+1}$. Taking the image of this relation under the maps $\widetilde{W_n}$ produces relation \eqref{eq:recursionina4} as follows.

From Proposition \ref{prop:recursion}, the recurrence relation for the symmetrizers in $BMW_{n+1}$ is
\begin{equation}
[n+1]f_{n+1} = [n]s^{-1}\left( f_n \otimes 1 \right) \left( 1 \otimes f_n \right) + \sigma_n \cdots \sigma_1 \left(  1 \otimes f_n \right) + [n]s^{-1}\beta_n\left( f_n \otimes 1 \right) h_n \cdots h_1 \left( 1 \otimes f_n \right) 
\end{equation}
where $\sigma_i$ is the standard simple braiding of the $i$ and $i+1$ strands and $h_i$ is the horizontal edge between the $i$ and $i+1$ strands. One may verify diagramatically that applying $\widetilde{W_n}$ to the above gives us relation \eqref{eq:recursionina4}:
\begin{align*}
[n+1]\widetilde{W}_{n} =& [n]s^{-1}\widetilde{W}_n\left( \left( f_n \otimes 1 \right) \left( 1 \otimes f_n \right) \right) + \widetilde{W}_n\left( \sigma_n \cdots \sigma_1 \left(  1 \otimes f_n \right) \right) + \\
&[n]s^{-1}\beta_n \widetilde{W}_n \left( \left( f_n \otimes 1 \right) h_n \cdots h_1 \left( 1 \otimes f_n \right) \right) \\
=&[n]s^{-1}\widetilde{W}_n\left( f_n \otimes 1 \right) + r_n + [n]s^{-1}\beta_n W_n \left( f_n \otimes 1 \right) \\
=&[n]s^{-1}a\widetilde{W}_{n-1} + r_n + [n]s^{-1}\beta_n a^{-1} W_{n-1}.
\end{align*}
This completes the proof. 
\end{proof}

We will also need the following identities for the $\beta_n$.

\begin{lemma}\label{lemma:ring}
The following relations hold in $\C(s, v)$:
\begin{gather*}
s - s^{-1}\beta_n = s^{-1} - s\bar{\beta}_n, \\
\left( \bar{\beta}_{n+1} - \beta_{n+1} \right) \left( s-s^{-1}\beta_n \right) = - \{1\}.
\end{gather*}
\end{lemma}

\begin{lemma}[\cite{She16}]\label{lemma:annfund}
For all $n\geq 1$, 
\begin{equation}
l_n-r_n = \{n\}\left( aW_{n-1} - a^{-1}\widetilde{W}_{n-1} \right).
\end{equation}
\end{lemma}
\begin{proof}
This is proven in Proposition 7 of [She16]. We can prove it here by taking the difference of equations \eqref{eq:recursionina1} and \eqref{eq:recursionina3} and using that $s^{-1}\beta_n - s\bar{\beta}_n = s-s^{-1}$, as provided by Lemma \ref{lemma:ring}.
\begin{align*}
l_n - r_n =& [n]\left( s-s^{-1} \right) aW_{n-1} - \left( s^{-1}\beta_n - s\bar{\beta}_n \right) a^{-1}\widetilde{W}_{n-1} \\
=& [n]\left( s-s^{-1} \right) \left( aW_{n-1} - a^{-1}\widetilde{W}_{n-1} \right) \\
=& \{n\}\left( aW_{n-1} - a^{-1}\widetilde{W}_{n-1} \right).
\end{align*}
\end{proof}
 
\begin{proposition}
The relations of Lemma \ref{lemma:perprel} hold.
\end{proposition}
\begin{proof}
In the case $n=-1$, the relation we want to show becomes 
\begin{equation*}
l_1-r_1 = \{1\} \left( a - a^{-1} \right).
\end{equation*}
This can easily be verified by applying the skein relation once to the left-hand side of the equation. 
By Lemma \ref{lemma:annfund}, the relation \eqref{eq:perprel3} is equivalent to:
\begin{equation}\label{eq_perprel4}
\{n+2\} \left( aW_{n+1} - a^{-1}\widetilde{W}_{n+1} \right) = \left( sa+s^{-1}a^{-1} \right) l_{n+1} - \left( sa^{-1}+s^{-1}a \right) r_{n+1} - \{n\}\left( aW_{n-1}-a^{-1}\widetilde{W}_{n-1} \right).
\end{equation}

We will show that the left hand side of the equation above may be reduced to the right hand side by a series of applications of Lemma \ref{lemma:recursionina}. 
\begin{align*}
& \{n+2\} \left( aW_{n+1} - a^{-1}\widetilde{W}_{n+1} \right) \\
\overset{4.5}{=}& \{n+2\} \left( \frac{a}{[n+2]} \left( l_{n+1} + [n+1]s^{-1}aW_n + [n+1]s^{-1}\beta_{n+1}a^{-1}\widetilde{W}_n \right) \right.- \\
&\left.\qquad\qquad\frac{a^{-1}}{[n+2]} \left( r_{n+1} + [n+1]s^{-1}a^{-1}\widetilde{W}_n + [n+1]s^{-1}\beta_{n+1}aW_n \right) \right)  \\
=& \left( s - s^{-1} \right) \left( \left( al_{n+1} + [n+1]s^{-1}a^2W_n + [n+1]s^{-1}\beta_{n+1}\widetilde{W}_n \right)- \right. \\
&\qquad\qquad\,\,\,\,\left.\left( a^{-1}r_{n+1} + [n+1]s^{-1}a^{-2}\widetilde{W}_n + [n+1]s^{-1}\beta_{n+1}W_n \right)\right) \\
\overset{4.5}{=}&\left( s-s^{-1} \right) \left( \left( al_{n+1} + s^{-2}a \left( [n+2]W_{n+1} -r_{n+1} - [n+1]s\bar{\beta}_{n+1}a^{-1}\widetilde{W}_n \right) + [n+1]s^{-1}\beta_{n+1}\widetilde{W}_n \right)- \right.  \\
&\qquad\qquad\,\,\,\, \left. \left( a^{-1}r_{n+1} + s^{-2}a^{-1}\left( [n+2]\widetilde{W}_{n+1} - l_{n+1} - [n+1]s\bar{\beta}_{n+1}aW_n \right) + [n+1]s^{-1}\beta_{n+1}W_n \right) \right) \\
=&\left( \left( sa + s^{-1}a^{-1} \right) l_{n+1} - \left( sa^{-1} + s^{-1}a \right) r_{n+1} \right) + \left( s^{-1}a^{-1} + s^{-3}a \right) r_{n+1} - \left( s^{-1}a + s^{-3}a^{-1} \right) l_{n+1} + \\
&\qquad\qquad\,\,\, \{n+2\}s^{-2}\left( aW_{n+1} - a^{-1} \widetilde{W}_{n+1} \right) + \{n+1\}s^{-1}\left( \bar{\beta}_{n+1} - \beta_{n+1} \right) \left( W_n - \widetilde{W}_n \right).
\end{align*}

We break the computation here to note that the first two terms in the last line also appear on the right hand side of (\ref{eq_perprel4}). Thus, we would like to prove the following equality:
\begin{align}\label{newgoal}
\begin{split}
&-\{n\} \left( aW_{n-1} - a^{-1}\widetilde{W}_{n-1} \right) = \left( s^{-1}a^{-1} + s^{-3}a \right) r_{n+1} - \left( s^{-1}a + s^{-3}a^{-1} \right) l_{n+1} + \\
&\qquad\qquad\qquad\qquad\qquad\qquad\quad \{n+2\}s^{-2}\left( aW_{n+1} - a^{-1} \widetilde{W}_{n+1} \right) - \{n+1\}s^{-1}\left( \bar{\beta}_{n+1} - \beta_{n+1} \right) \left( W_n - \widetilde{W}_n \right).
\end{split}
\end{align}

We will show this by using the identities from Lemma \ref{lemma:recursionina} on the right hand side of \ref{newgoal}. A large number of terms cancel and what remains is the desired identity.
\begin{align*}
& \left( s^{-1}a^{-1} + s^{-3}a \right) r_{n+1} - \left( s^{-1}a + s^{-3}a^{-1} \right) l_{n+1} + \{n+2\}s^{-2}\left( aW_{n+1} - a^{-1} \widetilde{W}_{n+1} \right) - \\
&\qquad\qquad \{n+1\}s^{-1}\left( \bar{\beta}_{n+1} - \beta_{n+1} \right) \left( W_n - \widetilde{W}_n \right) \\
=& \left( s^{-1}a^{-1} + s^{-3}a \right) r_{n+1} - \left( s^{-1}a+s^{-3}a^{-1} \right) l_{n+1} + [n+2]\left( s^{-1}-s^{-3} \right) \left(aW_{n+1} - a^{-1}\widetilde{W}_{n+1} \right) + \\
&\qquad\qquad [n+1]\left( 1-s^{-2} \right) \left( \bar{\beta}_{n+1}-\beta_{n+1} \right) \left(W_n - \widetilde{W}_n \right) \\
=& s^{-1}a\left( [n+2]W_{n+1} - l_{n+1} \right) + s^{-3}a^{-1}\left( [n+2]\widetilde{W}_{n+1} - l_{n+1} \right) - s^{-3}a\left( [n+2]W_{n+1} - r_{n+1} \right) - \\
&\qquad\qquad s^{-1}a^{-1}\left( [n+2]\widetilde{W}_{n+1} - r_{n+1} \right) + [n+1]\left( 1-s^{-2} \right) \left( \bar{\beta}_{n+1} -\beta_{n+1} \right) \left( W_n - \widetilde{W}_n \right) \\
\overset{4.5}{=}& s^{-1}a\left( [n+1]s^{-1}aW_n + [n+1]s^{-1}\beta_{n+1}a^{-1}\widetilde{W}_n \right) + s^{-3}a^{-1}\left( [n+1]sa^{-1}\widetilde{W}_n + [n+1]s\bar{\beta}_{n+1}aW_n \right) - \\
&s^{-3}a\left( [n+1]saW_n + [n+1]s\bar{\beta}_{n+1}a^{-1}\widetilde{W}_n \right) - s^{-1}a^{-1}\left( [n+1]s^{-1}a^{-1}\widetilde{W}_n + [n+1]s^{-1}\beta_{n+1}aW_n \right) + \\
& [n+1]\left(1-s^{-2} \right) \left( \bar{\beta}_{n+1} - \beta_{n+1} \right) \left( W_n - \widetilde{W}_n \right)\\
=& [n+1]\left( s^{-2}a^2W_n + s^{-2}\beta_{n+1}\widetilde{W}_n + s^{-2}a^{-2}\widetilde{W}_n + s^{-2}\bar{\beta}_{n+1}W_n - s^{-2}a^{2}W_n - s^{-2}\bar{\beta}_{n+1}\widetilde{W}_n - \right. \\
&\qquad\quad \,\, \left. s^{-2}a^{-2}\widetilde{W}_n - s^{-2}\beta_{n+1}W_n + \bar{\beta}_{n+1}W_n - \bar{\beta}_{n+1}\widetilde{W}_n - \beta_{n+1}W_n + \beta_{n+1}\widetilde{W}_n - \right. \\
&\qquad\quad\,\, \left. s^{-2}\bar{\beta}_{n+1}W_n + s^{-2}\bar{\beta}_{n+1}\widetilde{W}_n + s^{-2}\beta_{n+1}W_n - s^{-2}\beta_{n+1}\widetilde{W}_n \right) \\
=& [n+1]\left( \bar{\beta}_{n+1} - \beta_{n+1} \right) \left( W_n - \widetilde{W}_n \right) \\
=& \left( \bar{\beta}_{n+1} - \beta_{n+1} \right) \left( \left( [n+1]W_n \right) - \left( [n+1]\widetilde{W}_n \right) \right) \\
\overset{4.5}{=}& \left( \bar{\beta}_{n+1} - \beta_{n+1} \right) \left( \left( l_n + [n]s^{-1}aW_{n-1} + [n]s^{-1}\beta_{n}a^{-1}\widetilde{W}_{n-1} \right) - \left( l_n + [n]sa^{-1}\widetilde{W}_{n-1} + [n]s\bar{\beta}_{n}aW_{n-1} \right) \right) \\
=& \left( \bar{\beta}_{n+1} - \beta_{n+1} \right) \left( [n]\left( s^{-1} - s\bar{\beta}_{n} \right) aW_{n-1} - [n]\left( s-s^{-1}\beta_{n} \right) a^{-1}\widetilde{W}_{n-1} \right) \\
=& [n]\left(\bar{\beta}_{n+1} - \beta_{n+1} \right) \left( s - s^{-1}\beta_{n} \right) \left( aW_{n-1} - a^{-1}\widetilde{W}_{n-1} \right) \\
=& - \{n\}\left( aW_{n-1} - a^{-1}\widetilde{W}_{n-1} \right).
\end{align*}
Where the last equality follows from some simple algebra. This completes the proof.
\end{proof}

\section{Angled relations}\label{sec:angled}
In this section we prove the angled relations
\begin{equation}\label{eq:angleinsec}
[D_{1,0},D_{1,n}] = \{n\}(D_{2,n} - D_{0,n}).
\end{equation}
The argument here is more intricate than in the previous section, so we first give the main steps, and then prove the main two steps in independent subsections.

\begin{proof}[Proof of \eqref{eq:angleinsec}]
Let $a = [D_{1,0},D_{1,n}] - \{n\}(D_{2,n} - D_{0,n})$, so that our goal is to show $a=0$. Recall that $\cc$ is the skein of the annulus, and let $\cc_{0,1}$ be the image of $\cc$ when the annulus is embedded in $T^2$ along the $(0,1)$ curve. Also, recall that $\sd$ acts on $\cc$ 
as described in Remark \ref{rem:t2action}. In our conventions for this action, $(0,1) \in T^2$ corresponds to the core of the annulus, and $(1,0)$ is the meridian of the annulus. After making these choices, if we compose the algebra maps $\cc \to \cc_{0,1} \subset \sd \to \mathrm{End}(\cc)$, then the (commutative) algebra $\cc$ is acting on itself by multiplication.

The claimed relations then follow from the following statements:
\begin{enumerate}
	\item First, in Subsection \ref{sec:acting} we show that $a \cdot \varnothing = 0$, where $\varnothing $ is the empty link in $\cc$.
	\item Then, using skein-theoretic arguments, in Section \ref{sec:skein} we show that $a$ is in the subalgebra $\cc_{0,1}$.
	\item Finally, the following composition is an isomorphism
	\[
	\cc \to \cc_{0,1} \to \cc
	\]
	where the second map is given by $b \mapsto b\cdot \varnothing$.
	This follows from the fact that there is a deformation retraction comprised of smooth embeddings which takes $S^1 \times D^2$ to the image of a neighborhood of the $(0,1)$ curve under the gluing map 
	\[
	\left( T^2\times [0,1] \right) \sqcup \left( S^1 \times D^2 \right) \to S^1 \times D^2.
	\]
\end{enumerate}
Combining these three facts shows $a=0$ in the algebra $\sd$.
\end{proof}

\subsection{Acting on the empty link}\label{sec:acting}
Here we partially describe the action of the skein algebra of the torus on the skein algebra of the annulus. The formula we give here is similar in spirit to the formula for torus link invariants in \cite{RJ93}, but it is simpler because of our choice to color the torus link by $D_k$. (For uncolored torus links there are idempotents other than hooks on the right hand side of the formula in \cite{RJ93}.) However, we note that we don't actually use the results of \cite{RJ93} since they are phrased in terms of evaluations in $S^3$. 

\begin{lemma}\label{lemma:projection}
Suppose that $k \in \N_{\geq 1}$, that $m,n \in \Z$ are relatively prime, and that $n \geq 1$. Then
\[
	D_{km,kn}\cdot\varnothing = c_k + \sum_{a+b+1=kn} v^{-km} s^{km(a-b)} (-1)^b \widehat{Q}_{(a|b)}.
\]
\end{lemma}
\begin{proof}
The left hand side is the $(km,kn)$ torus link colored by the element $D_k$, and projected into the skein of the annulus. Since the left hand side of Corollary \ref{cor:projection} is an explicit expression in the BMW algebra for an element which closes to the $B_k$-colored $(km,kn)$ torus link in the annulus, the formula claimed in this lemma follows from Corollary \ref{cor:projection} and equation \eqref{eq:dvsb}.
\end{proof}

Since the constants $c_k$ from equation \eqref{eq:ck} only depend on the parity of $k$, and the parity of $\gcd(2,n)$ and $\gcd(0,n)$ is the same, Lemma \ref{lemma:projection} implies the following identity:
\begin{equation}\label{eq:projection1}
	\{n\}\left( D_{2,n} - D_{0,n} \right) \cdot \varnothing = \{n\} \sum_{a+b+1=n} \left( v^{-2}s^{2(a-b)} - 1 \right) (-1)^b \widehat{Q}_{(a|b)}.
\end{equation}

By \cite[Proposition 2.1]{LZ02}, $\widehat{Q}_{\lambda}$ is an eigenvector of the action by $D_{1,0}$ on $\cc$ with eigenvalue 
\[
	c_{\lambda} = \delta + \left( s - s^{-1} \right) \left( v^{-1} \sum_{x \in \lambda} s^{2cn(x)} - v \sum_{x \in \lambda} s^{-2cn(x)} \right)
\]
where if $x$ is a cell of index $(i,j)$ in $\lambda$, it has \emph{content} $cn(x)= j - i$. When $\lambda = (a|b)$, one may compute
\[
	c_{(a|b)} = \delta + v^{-1} \left( s^n - s^{-n} \right) s^{a-b} - v \left( s^n - s^{-n} \right) s^{-(a-b)}.
\]
This together with \eqref{eq:projection1} gives us
\begin{eqnarray*}
	[D_{1,0},D_{1,n}] \cdot \varnothing &=& D_{1,0} \cdot \left( D_{1,n} \cdot \varnothing \right) - D_{1,n} \cdot \left( D_{1,0} \cdot \varnothing \right) \\
	&=& D_{1,0} \cdot \left( \sum_{a+b+1=n} v^{-1} s^{a-b} (-1)^b \widehat{Q}_{(a|b)} \right) - D_{1,n} \cdot \left( \delta \varnothing \right) \\
	&=& \sum_{a+b+1=n} \left( c_{(a|b)} - \delta \right) v^{-1} s^{a-b} (-1)^b \widehat{Q}_{(a|b)} \\
	&=& \sum_{a+b+1=n} \left( v^{-1} \left( s^n - s^{-n} \right) s^{a-b} - v \left( s^n - s^{-n} \right) s^{-(a-b)} \right) v^{-1} s^{a-b} (-1)^b \widehat{Q}_{(a|b)} \\
	&=& \{n\} \sum_{a+b+1=n} \left( v^{-2}s^{2(a-b)}- 1 \right) (-1)^b \widehat{Q}_{(a|b)}\\
	&=& \{n\} (D_{2,n}\cdot \varnothing -D_{0,n} \cdot \varnothing).
\end{eqnarray*}

This shows that $a \cdot \varnothing = 0$.

\subsection{The affine BMW algebra}\label{sec:skein}

In this section, we will show that the element $a \in \sd$ lies in $\cc_{0,1}$. We will do this in the following skein-based calculations, starting with the use of the affine BMW algebra on 2 strands, $\dot{BMW}_2$.  The algebra $\dot{BMW}_n$ is isomorphic to the algebra of $n$-tangles in the thickened annulus, modulo the Kauffman skein relations, where the product is stacking digrams \cite{GH06}. We will think of the diagrams as living in the square with the top and bottom horizontal edged identified, where a product of diagrams $XY$ is the diagram obtained from connecting the right vertical edge of $X$ with the left vertical edge of $Y$.

Making a further identification of the vertical edges allows us to pass to the skein of the torus $T^2$, and induces a linear map $cl_2:\dot{BMW}_2\to\sd$. This closure map $cl_2$ has the standard `trace' property, namely that $cl_2(XY)=cl_2(YX)$ for any $X,Y\in\dot{BMW}_2$. Similar to $cl_1$, this map may also be defined using a careful choice of wiring diagram:
\[
\vcenter{\hbox{\includegraphics[width=6cm]{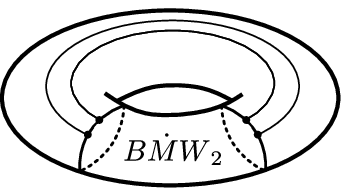}}}.
\]

From the diagrams $A$ and $\bar A$ as shown: 
\[
A = \vcenter{\hbox{\includegraphics[width=3cm]{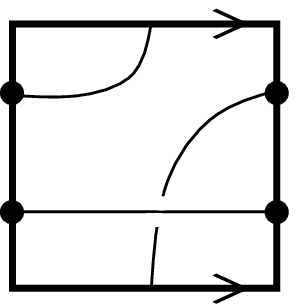}}} \quad \qquad \qquad \bar{A} = \vcenter{\hbox{\includegraphics[width=3cm]{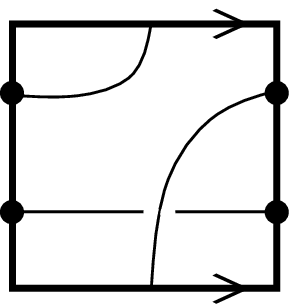}}}
\]
we can picture $A^n$ and $\bar{A}^n$.
The closures of  these  satisfy $cl_2(A^n)=D_{1,0}D_{1,n}$  and $cl_2(\bar{A}^n)=D_{1,n}D_{1,0}$. Hence \[[D_{1,0},D_{1,n}]=cl_2(A^n-\bar{A}^n).\]

By the skein relation we can write $A-\bar{A}=(s-s^{-1})(C-D)$ in $\dot{BMW}_2$, where $C$ and $D$ are the diagrams:
\[
C = \vcenter{\hbox{\includegraphics[width=3cm]{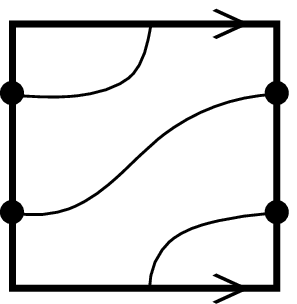}}} \quad \qquad \qquad D = \vcenter{\hbox{\includegraphics[width=3cm]{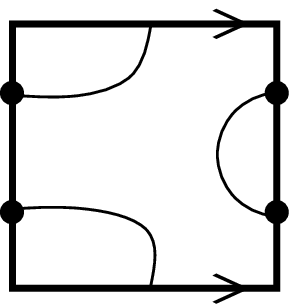}}}
\]

Using Definition \ref{def:D_x} and the explicit computation of $D_2$ in \eqref{eq:b2}, one can verify that 
\begin{equation}\label{eq:deven}
	D_{2,n}=
	\begin{cases}
 	cl_2 (C^n) & n \textrm{ odd} \\
      	cl_2 \left( \frac{1}{s+s^{-1}} \left( A C^{n-1}+\bar A C^{n-1} \right) \right) - c & n \textrm{ even} \\
	\end{cases}
\end{equation}
where 
\[
	c=\frac{v+v^{-1}}{s+s^{-1}} - 1.
\]

The product of $D$ with any element of $\dot{BMW}_2$  closes in $T^2$ to an element in the skein of $T^2$ lying entirely in an  annulus around the vertical curve, $(0,1)$:
\[
 \vcenter{\hbox{\includegraphics[width=6cm]{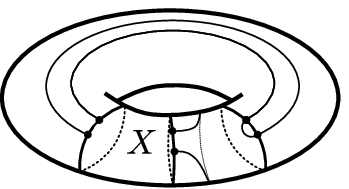}}}
\]

In algebraic terms,
$cl_2(XD)\in\cc_{0,1}$ for any $X\in\dot{BMW}_2$. Write $\si$ for the two-sided ideal of $\dot{BMW}_2$ generated by $D$. Then $cl_2(\si)\subset \cc_{0,1}$.

Note that $D_{0,n}$ is in $\cc_{0,1}$ by definition. Our argument then, for $n$ odd, is to show that $cl_2(A^n-\bar{A}^n-(s^n-s^{-n})C^n) \in  \cc_{0,1}$, with the appropriate modification for $n$ even. In the case where $n$ is even, we may forget about the constant term in $D_{2,n}$ since $-c\in\cc_{0,1}$.

A few observations in the algebra $\dot{BMW}_2$ will get us much of the way towards our goal.
Firstly we can see that $C^2$ commutes with $A$ and $\bar{A}$ at the diagrammatic level. Indeed it is central in the algebra, and 
so commutes with $D$, although more strongly we have that $C^2 D=DC^2=D$.  
We see also that $AC\bar{A}=C^3 =\bar ACA$, leading to the more general cancellation of crossings $AC^{2k-1}\bar{A}=C^{2k+1}$ for $k>0$, again true at the diagrammatic level.

\begin{lemma}
 In $\dot{BMW}_2$ we have 
\[
	A^n-\bar{A}^n=(s-s^{-1})\left(A^{n-1}C +A^{n-3}C^3 + \dots + C^3\bar{A}^{n-3}+C\bar{A}^{n-1} -\sum_{i=1}^n A^{n-i}D\bar{A}^{i-1}\right).
\]
\end{lemma}
\begin{proof}
\begin{eqnarray*}
A^n-\bar{A}^n &=& \sum_{i=1}^n A^{n-i}(A-\bar{A})\bar{A}^{i-1}\\
&=& (s-s^{-1})\sum_{i=1}^n A^{n-i}(C-D)\bar{A}^{i-1}\\
&=& (s-s^{-1})\left(A^{n-1}C +A^{n-3}C^3 + \dots + C^3\bar{A}^{n-3}+C\bar{A}^{n-1} -\sum_{i=1}^n A^{n-i}D\bar{A}^{i-1}\right)
\end{eqnarray*}
\end{proof}

\begin{remark}
 The terms in a product can be reordered cyclically before closure so that 
 \[
	cl_2 \left(\sum_{i=1}^n A^{n-i}D\bar{A}^{i-1}\right) = \cl_2 \left(\sum_{i=1}^n \bar{A}^{i-1}A^{n-i}D\right)= X_n
\]
  which represents the element $X_n\in\cc_{0,1}=\cc$ in the skein of the annulus which has the same diagrammatic representation as the $P_k$ in the Hecke algebra in Definition \ref{def:pk}.
\end{remark}

\begin{corollary}\label{expansion} 
	We have the expansion
	\begin{align*}
	[D_{1,0},D_{1,n}] &= cl_2 (A^n-\bar A^n)\\
	&= cl_2 \left((s-s^{-1})(A^{n-1}C+C\bar A^{n-1})+ (A^{n-2}-\bar A^{n-2})C^2\right) -(s-s^{-1})(X_n-X_{n-2}).
	\end{align*}
\end{corollary}

\begin{theorem}\label{commutator}
 In $\sd$ we have $[D_{1,0},D_{1,n}] =(s^n-s^{-n})D_{2,n}+Z_n$ for some  $Z_n\in \cc_{0,1}$.
\end {theorem}
\begin{proof} We prove this by induction on $n$ in two separate cases, either when $n$ is odd, or when $n$ is even.

The case $n=1$ is immediate since $A-\bar A-(s-s^{-1})C=-(s-s^{-1})D\in\si$ and so $cl_2 (A-\bar A-(s-s^{-1})C)= [D_{1,0},D_{1,1}] - (s^1-s^{-1})D_{2,1}$ lies in $\cc_{0,1}$.

When $n=2$ we have from  Lemma \ref{expansion} that $A^2-\bar A^2 -(s-s^{-1})(AC +C\bar A) \in \si$. Then  $cl_2 (A^2-\bar A^2 -(s-s^{-1})(AC +C\bar A))=[D_{1,0},D_{1,2}] - (s^1-s^{-1})(s+s^{-1})D_{2,2}$ lies in $\cc_{0,1}$.\\[2mm]

Assume now that $n>2$.
By our induction hypothesis, Theorem \ref{commutator}, for $n-2$ we know that 
\begin{eqnarray*}[D_{1,0},D_{1,n-2}] &=& (s^{n-2}-s^{-(n-2)})D_{2,n-2} +Z_{n-2},
\end{eqnarray*} where $Z_{n-2}\in\cc_{0,1}$. 
Apply a single Dehn twist about the $(0,1)$ curve. Then 
\[
	[D_{1,1},D_{1,n-1}]= (s^{n-2}-s^{-(n-2)})D_{2,n} +Z_{n-2},
\] 
since $\cc_{0,1}$ is unaffected  by this Dehn twist. 
We can see diagrammatically that 
\[
	cl_2 (A^{n-2}C^2-\bar A^{n-2}C^2)=[D_{1,1},D_{1,n-1}].
\]
Our induction hypothesis then shows that
\begin{eqnarray}\label{commutatorcentre} 
	cl_2 (A^{n-2}C^2-\bar A^{n-2}C^2)&=&(s^{n-2}-s^{-(n-2)})D_{2,n}\ \rm{modulo}\  \cc_{0,1}.
\end{eqnarray}
(Here and below, when we write ``modulo'' we mean as an $R$-linear subspace.)
This provides us with one part of the expression for $[D_{1,0},D_{1,n}] $ in Corollary \ref{expansion}. We calculate the other part by means of the following lemma.

\begin{lemma}\label{endpart} We have
$cl_2 \left(A^{n-1}C+C\bar A^{n-1}\right)=(s^{n-1}+s^{-(n-1)})D_{2,n} + Y_n$
for some $Y_n\in\cc_{0,1}$.
\end{lemma}
\begin{proof}
Lemma \ref{endpart} is immediate for $n=1,2$. 

For $n>2$  use induction on $n$, and equation (\ref{commutatorcentre}). 

Now $ cl_2 (A^{n-1}C+C\bar A^{n-1}) = cl_2 (A^{n-2}CA+\bar AC\bar A^{n-2}).$
Expand $A^{n-2}CA$ and $\bar AC\bar A^{n-2}$ as
 \[
	A^{n-2}CA=A^{n-2}C\left(\bar A+(s-s^{-1})(C-D)\right)=A^{n-3}C^3 +(s-s^{-1})A^{n-2}C^2 -(s-s^{-1})A^{n-2}CD
\] 
and
\[
	\bar A C\bar A^{n-2}= \left( A-(s-s^{-1})(C-D)\right)C\bar A^{n-2}=C^3\bar A^{n-3} -(s-s^{-1})\bar A^{n-2}C^2 +(s-s^{-1})DC\bar A^{n-2}.
\]

Then 
\begin{eqnarray*}  
	cl_2 (A^{n-2}CA+\bar A C\bar A^{n-2})
	&=& cl_2 \left(A^{n-3}C^3+C^3\bar A^{n-3}+(s-s^{-1})(A^{n-2}C^2-\bar A^{n-2}C^2) \right)\ \rm{modulo}\  \cc_{0,1}\\
	&=& (s^{n-3}+s^{-(n-3)})D_{2,n}+(s-s^{-1})(s^{n-2}-s^{-(n-2)})D_{2,n}\ \rm{modulo}\  \cc_{0,1}\\
	&=& (s^{n-1}+s^{-(n-1)})D_{2,n} \ \rm{modulo}\  \cc_{0,1}, 
\end{eqnarray*}
using Lemma \ref{endpart} for $n-2$ and equation (\ref{commutatorcentre}).
This establishes Lemma \ref{endpart}.
\end{proof}
We complete the induction step for Theorem \ref{commutator} using  Corollary \ref{expansion} and equation (\ref{commutatorcentre}).
This shows that 
\begin{eqnarray*} 
	[D_{1,0},D_{1,n}]&=& cl_2 \left((s-s^{-1})(A^{n-1}C +C\bar{A}^{n-1}) + (A^{n-2}-\bar A^{n-2})C^2\right)  \ \rm{modulo}\  \cc_{0,1}\\
	&=& (s-s^{-1})(s^{n-1}-s^{-(n-1)})D_{2,n} +(s^{n-2}-s^{-(n-2)})D_{2,n}\ \rm{modulo}\  \cc_{0,1}\\
	&=&(s^n-s^{-n})D_{2,n} \ \rm{modulo}\  \cc_{0,1}.
\end{eqnarray*}

\end{proof}

%
%
%
%
%
%

\section{Appendix}
In the first subsection we show that our results on the Kauffman skein algebra are compatible with the analogous result about the Kauffman \emph{bracket} skein algebra obtained by Frohman and Gelca \cite{FG00}. In the second subsection we state the BMW analogue of a classical result of Przytycki.


\subsection{Frohman and Gelca}
Here we show that our presentation of the Kauffman skein algebra of the torus is compatible with Frohman and Gelca's presentation of the Kauffman \emph{bracket} skein algebra of the torus, which we recall below. This algebra is defined in the same way as the Kauffman skein algebra, but using the Kauffman bracket skein relations:
\begin{align}\label{eq:kb}
\vcenter{\hbox{\includegraphics[height=2cm]{poscross.eps}}}  \quad &= \quad  s \,\,\vcenter{\hbox{\includegraphics[height=2cm]{idresolution.eps}}}  + s^{-1} \,\, \vcenter{\hbox{\includegraphics[height=2cm]{capcupresolution.eps}}}\\
\vcenter{\hbox{\includegraphics[height=2cm, keepaspectratio]{invvh.eps}}} \quad &= \quad -s^{3}\,\, \vcenter{\hbox{\includegraphics[height=2cm, keepaspectratio]{frameresolution.eps}}}. \notag
\end{align}

Given any 3-manifold, there is a natural map $\sk_s(M) \to K_s(M)$, which follows from the fact that the Kauffman bracket skein relations imply the Kauffman skein relations. If $M = F \times [0,1]$ for some surface $F$, then this is an algebra map.


The Frohman and Gelca presentation of the Kauffman bracket skein algebra uses the Chebyshev polynomials $T_n(x)$, which are uniquely characterized by the property $T_n(X+X^{-1}) = X^n + X^{-n}$. If $\xx \in \Z^2$ with $d(\xx) = k$, they defined an element $e_\xx \in K_s(T^2)$ by 
\[
e_\xx := T_k(\alpha_{\xx / k})
\]
where $\alpha_\xx$ is the simple closed curve of slope $\xx$.

\begin{lemma}
	The algebra map $\sk_s(T^2\times [0,1]) \to K_s(T^2\times [0,1])$ takes the generator $D_\xx$ to  $e_\xx$.
\end{lemma}
\begin{proof}
The closure of the symmetrizer $f_n$ in the annulus gets sent to the (closure of the) Jones-Wenzl idempotent in the Kauffman bracket skein algebra of the annulus, because both symmetrizers are uniquely determined by the way they absorb crossings and caps. Under the identification $K_s(S^1\times D^2) = \C[x]$, the closure of the Jones-Wenzl idempotent is sent to $S_n(x)$, the other version of the Chebyshev polynomial, which is defined by 
\[S_n(X+X^{-1}) = \frac{X^{n+1}-X^{-n-1}}{X-X^{-1}}.\]
All that is left to show is that the $S_n(x)$ and $T_n(x)$ satisfy the power series identity in Definition \ref{def:dk}, which we write as
\begin{equation*}
\sum_{k=1}^\infty \frac{T_k(x)}{k} t^k =\ln\left(1 + \sum_{j \geq 1} S_j(x) t^j\right).
\end{equation*}	
Define different variants of  Chebyshev polynomials via $C_n^S(x/2):=S(x)$ and $2C_n^T(x/2):=T(x)$. The following two identities are well known (see, e.g. any book on special functions, or Wikipedia):
\begin{align*}
	\sum_{n \geq 0} C_n^S(x)t^n &= \frac{1}{1-2tx+t^2} \\
	\sum_{n \geq 1} C_n^T(x)\frac{t^n}{n} &= \ln \left( \frac{1}{\sqrt{1-2tx+t^2}} \right).
\end{align*}
The first identity implies 
\[
	1+\sum_{n \geq 1} S_n(x)t^n = \frac{1}{1-tx+t^2}.
\]
This implies the follow identities:
\begin{equation*}
	1+\sum_{n \geq 1} S_n(x)t^n  = \frac{1}{1-tx+t^2} = \left( \exp \left( \sum_{n \geq 1} C_n^T(x/2)\frac{t^n}{n} \right) \right)^2
	= \exp \left( \sum_{n \geq 1} 2 C_n^T(x/2)\frac{t^n}{n} \right).
\end{equation*}
This leads to the following, which completes the proof:
\begin{equation*}
\ln \left( 1+\sum_{n \geq 1} S_n(x)t^n \right) = \sum_{n \geq 1} 2 C_n^T(x/2)\frac{t^n}{n} = \sum_{n \geq 1} \frac{T_n(x)}{n} t^n.
\end{equation*} 
\end{proof}

Now let us recall the description of the Kauffman bracket skein algebra of the torus given by Frohman and Gelca.
\begin{theorem}[{\cite{FG00}}]\label{thm:fg}
	The Kauffman bracket skein algebra $K_s(T^2)$ has a presentation with generators $e_\xx$ and relations
	\begin{align}
	e_\xx e_\yy &= s^d e_{\xx + \yy} + s^{-d} e_{\xx-\yy}\label{eq:fgrel}\\
	e_{\xx} &= e_{-\xx}\notag
	\end{align}
	where $d = \det[\xx\,\yy]$.
\end{theorem}

A short computation using this presentation shows  the commutator identity
\begin{equation}\label{eq:fg}
[e_\xx, e_\yy] = (s^d-s^{-d}) (e_{\xx+\yy} - e_{\xx-\yy}).
\end{equation}
In other words, the relations we show for $D_\xx \in \sk(T^2)$ get sent to the relations in \eqref{eq:fg}, and the relations Frohman and Gelca showed in Theorem \ref{thm:fg} for the $e_\xx \in K_s(T^2)$ imply the relations in \eqref{eq:fg}. This shows our results are compatible with those of Frohman and Gelca.

\begin{remark}\label{rmk:sobig}
We emphasize that the algebra we describe is much bigger -- $\sk(T^2)$ has a linear basis given by unordered words in the $D_\xx$, while the Kauffman bracket skein algebra $K_s(T^2)$ has a linear basis given by the $e_\xx$ themselves. This happens because the Kauffman bracket skein relations allow all crossings to be removed from a diagram, so the algebra is spanned (over $R$) by curves without crossings. However, the more general Kauffman relations only allow crossings to be flipped, and not removed, which means the algebra is spanned (over $R$) by products of knots (which may have self-crossings). In particular, the relations \eqref{eq:fgrel} do not hold in the Kauffman skein algebra $\sd$.
\end{remark}

\subsection{Przytycki}
In \cite{Prz92}, Przytycki described a linear basis of the Homflypt skein algebra of a surface. This basis is the set of unordered words on the set $C := \{\pi_1(F) \setminus 1\}/\sim$, where the equivalence relation is $g \sim hgh^{-1}$. The  set $C$ parameterizes homotopy classes of nontrivial oriented curves. For the torus, this implies that the Homflypt skein algebra has a basis given by unordered words in $P_\xx$ for $\xx \in \Z^2$. A rough statement of the idea of the proof is that the diamond-lemma-type equalities that are used to show that the Homflypt polynomial is well-defined in $S^3$ generalize to 3-manifolds of the form $F \times [0,1]$. 
\PS{changed a sentence here}

The analogue of Przytycki's theorem in our case involves homotopy classes of unoriented curves, and we give the statement without proof.
\begin{theorem}\label{thm:basis}
	The Kauffman skein algebra $\sd(F \times [0,1])$ of a surface $F$ has a basis given by unordered words on the set $\{ \pi_q(F) \setminus 1\} / \sim$, where the equivalence relation is generated by the relations $g \sim hgh^{-1}$ and $g \sim g^{-1}$. In particular, $\sd(T^2 \times [0,1])$ is has a basis given by unordered words in $\{ D_\xx \mid \xx \in \Z^2/\sim,\,\, \xx \sim -\xx \}$, subject to the relation $D_\xx = D_{-\xx}$.
\end{theorem}


\bibliography{somerefs}{}
\bibliographystyle{amsalpha}

\end{document}